\documentclass[12pt]{amsart}
\usepackage{fullpage}
\usepackage{graphicx}
\usepackage{todonotes}
\usepackage{tikz}
\usetikzlibrary{decorations.markings}

\newtheorem{theorem}{Theorem}[section]

\newtheorem{lemma}[theorem]{Lemma}
\newtheorem{proposition}[theorem]{Proposition}
\newtheorem{definition}[theorem]{Definition}

\newtheorem{remark}[theorem]{Remark}

\numberwithin{equation}{section}

\newcommand\NN {{\mathbb N}}
\newcommand\QQ {{\mathbb Q}}
\newcommand\RR {{\mathbb R}}
\newcommand\TT {{\mathbb T}}
\newcommand\ZZ {{\mathbb Z}}

\newcommand\slgroup{{\rm SL}}

\newcommand\unit {{\mathbf 1}}

\newcommand\distance{{\rm Dist }}

\newcommand\interior{{\rm Int}}
\newcommand\closure{{\rm Cl}}

\newcommand\cA{{\mathcal{A}  }}

\newcommand\cC{{\mathcal{C}  }}
\newcommand\cD{{\mathcal{D}  }}

\newcommand\cF{{\mathcal{F}  }}
\newcommand\cG{{\mathcal{G}  }}

\newcommand\cI{{\mathcal{I}  }}

\newcommand\cO{{\mathcal{O}  }}
\newcommand\cP{{\mathcal{P}  }}
\newcommand\cQ{{\mathcal{Q}  }}
\newcommand\cR{{\mathcal{R}  }}

\newcommand\cT{{\mathcal{T}  }}
\newcommand\cU{{\mathcal{U}  }}
\newcommand\cV{{\mathcal{V}  }}

\newcommand\cZ{{\mathcal{Z}  }}

\begin{document}

\title[Generalized IET]{Full families of generalized Interval Exchange Transformations}

\author[L. Marchese]{Luca Marchese}
\address{Universit\'e Paris 13, Sorbonne Paris Cit\'e,
LAGA, UMR 7539, 99 Avenue Jean-Baptiste Cl\'ement, 93430 Villetaneuse, France.}

\email{marchese@math.univ-paris13.fr}

\author[L. Palmisano]{Liviana Palmisano}
\address{University of Bristol, Howard House, Queens Ave.,
Bristol BS8 1SD, U.K.}

\email{liviana.palmisano@gmail.com}


\begin{abstract}
We consider generalized interval exchange transformations, or briefly GIETs, that is bijections of the interval which are piecewise increasing homeomorphisms with finite branches. When all continuous branches are translations, such maps are classical interval exchange transformations, or briefly IETs. The well-known Rauzy renormalization procedure extends to a given GIET and a Rauzy renormalization path is defined, provided that the map is infinitely renormalizable. We define full families of GIETs, that is optimal finite dimensional parameter families of GIETs such that any prescribed Rauzy renormalization path is realized by some map in the family. In particular, a GIET and a IET with the same Rauzy renormalization path are semi-conjugated. This extends a classical result of Poincar\'e relating circle homeomorphisms and irrational rotations.  
\end{abstract}

\maketitle


\section{Introduction}

In 1885 Poincar\'e proved that any circle homeomorphism with irrational \emph{rotation number} is semi-conjugate to a rotation. More precisely, consider $\TT:=\RR/\ZZ$ and a homeomorphism $f:\TT\to\TT$ with irrational rotation number $\alpha$. Then there exists a continuous map 
$
h:\TT\to\TT
$ 
which preserves the cyclic order and such that $h\circ f=R_\alpha\circ h$, where $R_\alpha:\TT\to\TT$ is the map given by 
$
R_\alpha(x):=x+\alpha\mod\ZZ
$, 
see \cite{Poincare}. If the conjugacy $h$ is a homeomorphism, then $f$ and $R_\alpha$ are conjugate. Denjoy proved that this holds if $f$ is a $C^1$ diffeomorphism and, as essential additional assumption, its derivative $f'$ has bounded variation, see \cite{Denjoy}. 
More refined results on the smoothness of the conjugacy $h$ where obtained later under specific assumptions on the rotation number, see in particular \cite{KatzOrnst}, \cite{Hermann}, \cite{Yoccoz}.



It is natural to consider these questions for the class of \emph{generalized interval exchange transformations}, or briefly GIETs, that is bijections $f:[0,1)\to[0,1)$ of the interval which are locally piecewise increasing homeomorphisms, where the number of continuous branches is finite (see \S~\ref{SectionStatementOfResults} for definitions). These maps arise as a natural generalization of \emph{interval exchange transformations}, or briefly IETs, which are the maps in the class described above for which all continuous branches are translations (our way to renormalize GIETs is modeled on \emph{Rauzy induction} on classical IETs, and for this reason we restrict to maps whose continuous branches are strictly increasing). In this setting, we consider extensions of Poincar\'e's result. First of all, the notion of rotation number is generalized to GIETs by the notion of \emph{Rauzy renormalization path} (see Definition~\ref{DefinitionRauzyPathOfGIET}), where in particular infinitely renormalizable GIETs are those for which such a renormalization path exists and is \emph{infinite complete} in the sense of Marmi-Moussa-Yoccoz \cite{MarmiMoussaYoccozCohomological}. If 
$f:[0,1)\to[0,1)$ is an infinitely renormalizable GIET and $T:[0,1)\to[0,1)$ is an IET with the same renormalization path, then the two maps are semi-conjugated, that is $h\circ f=T\circ h$ for a continuous non-decreasing map $h:[0,1)\to[0,1)$ (see Proposition~\ref{PropositionSemiconjugation}).  A consequent question is to ask if any Rauzy renormalization path is in fact realized in the class of GIETs. Our main result, namely Theorem~\ref{TheoremFullFamily}, gives a positive answer, under a mild combinatorial assumption which plays an essential role in our argument (a discussion on the generality of such assumption appears in Remark~\ref{RemarkExistenceCyclicData}). More precisely, in Definition~\ref{DefinitionFullFamily} we introduce parameter families of GIETs with the optimal finite number of real parameters. When the extra combinatorial assumption is satisfied, then, for any specific Rauzy renormalization path, there exists a map in the family which realizes it. Having such a property, these parameter families of GIETs are called \emph{full families}. 

Similar results on full families have been proved in the past in different settings: continuous non decreasing circle maps \cite{deMelo-vanStrien}, tent maps \cite{ Milnor-Thurston}, multimodal interval maps \cite{Milnor-Thurston}, Lorenz maps \cite{Martens-deMelo}, quadratic complex maps \cite{DuadyHubbard}. A negative result occurs for H\'enon maps \cite{HazardMartensTresser}. For circle maps the full family theorem is a direct consequence of the continuity of the rotation number 
and the mean value theorem. Complex methods are used for proving the theorem in the context of the quadratic complex family. In the other cases the result is achieved using a method introduced by Thurston. In \S~\ref{SectionThurstonMap} and \S~\ref{SectionExistenceFixedPointThurstonMap} we adapt such method to the setting of GIETs: we define a map on the \emph{space of configurations} (see Definition~\ref{DefinitionGammaConfiguration}), which is naturally identified with a finite dimensional simplex, and we obtain a dynamically generated configuration as a fixed point of such map. Such dynamically generated configuration corresponds to a GIET in the full family. Then, considering longer and longer finite renormalization paths and taking a limit one obtains the required GIET. The properties (1), (2) and (3) defining a full family (see Definition~\ref{DefinitionFullFamily}) arise naturally while we implement such method.

Our results are the beginning of a long-term project aiming to describe the quality of the semi-conjugacy between a GIET and a classical IET. It is known that already for the case of affine interval exchange transformations the semi-conjugacy is not 
always a conjugacy, see \cite{MarmiMoussaYoccozWandering}. The renormalization path generated by the map is not the only invariant determining whether systems are conjugated or not. We expect that similar and new phenomena will appear also for GIETs. Our result is also the first step for describing the 
semi-conjugacy classes. It is conjectured in \cite{MarmiMoussaYoccozLinearization} that they are topological submanifolds of co-dimension $d-1$, where $d$ is the number of the subintervals determining the IET. In the next section we present our results in all details.

\subsection{Statement of results}
\label{SectionStatementOfResults}

An alphabet $\cA$ is a finite set with cardinality $|\cA|=d\geq 2$. A combinatorial datum over $\cA$ is a pair 
$
\pi=(\pi_t,\pi_b)
$ 
of bijections 
$
\pi_t,\pi_b:\cA\to\{1,\dots,d\}
$. 
Such $\pi$ is said \emph{admissible} if 
$
\pi_t^{-1}\{1,\dots,k\}\not=\pi_b^{-1}\{1,\dots,k\}
$ 
for any $k=1,\dots,d-1$. A \emph{right-open interval} is an interval $[a,b)\subset\RR$ closed on the left and open on the right. Denote $\interior(I):=(a,b)$ and $\closure(I):=[a,b]$ respectively its interior and its closure. Obviously, right-open intervals are never empty. When there is no ambiguity, we will refer to them simply as intervals.

\medskip

Given a combinatorial datum $\pi$ and a right-open interval $I$, consider two partitions of $I$ into $d$ right-open subintervals
$
\cP_t=\{I^t_\alpha\}_{\alpha\in\cA}
$ 
and 
$
\cP_b=\{I^b_\alpha\}_{\alpha\in\cA}
$, 
where for any $\alpha\in\cA$ the subintervals   
$
I_\alpha^t
$ 
and 
$
I_\alpha^b
$ 
appear respectively in $\pi_t(\alpha)$-th and $\pi_b(\alpha)$-th position in $I$ counting from the left.   A \emph{generalized interval exchange transformation}, of briefly GIET, is a map $f:I\to I$ such that for any $\alpha\in\cA$ the restricted map
\begin{equation}
\label{EquationPiecewiseHomeomorphism}
f|_{I_\alpha^t}:I_\alpha^t\to I_\alpha^b
\end{equation}
is an orientation preserving homeomorphism (see the left part of Figure~\ref{FigureDegenerationGIETs}). The restricted map in Equation~\eqref{EquationPiecewiseHomeomorphism} is called a \emph{continuous branch} of $f$ and is denoted $f_\alpha$. Clearly, if $f:I\to I$ is a GIET, then $f$ is a bijection of $I$ and its inverse $f^{-1}:I\to I$ is also a GIET. A map $T:I\to I$ is an \emph{interval exchange transformation}, or briefly IET, if it is a GIET and for any $\alpha\in\cA$ the restricted map in Equation~\eqref{EquationPiecewiseHomeomorphism} is a translation. A \emph{length datum} is any vector $\lambda\in\RR_+^\cA$ with all entries positive. Any pair of combinatorial-length data $(\pi,\lambda)$ determines uniquely an IET, denoted $T=T(\pi,\lambda)$, acting on the interval 
$
I:=[0,\sum_{\chi\in\cA}\lambda_\chi)
$, 
where the intervals in the partitions  
$
\cP_t=\{I^t_\alpha\}_{\alpha\in\cA}
$ 
and 
$
\cP_b=\{I^b_\alpha\}_{\alpha\in\cA}
$ 
of $I$ are defined for $\alpha\in\cA$ by
$$
I^t_\alpha:=
\Big[
\sum_{\pi_t(\chi)\leq\pi_t(\alpha)-1}\lambda_\chi,
\sum_{\pi_t(\chi)\leq\pi_t(\alpha)}\lambda_\chi
\Big)
\quad
\textrm{ and }
\quad
I^b_\alpha:=
\Big[
\sum_{\pi_b(\chi)\leq\pi_b(\alpha)-1}\lambda_\chi,
\sum_{\pi_b(\chi)\leq\pi_b(\alpha)}\lambda_\chi
\Big).
$$ 

We denote by $\cG(\pi,I)$ the set of all GIETs defined on the interval $I$ with combinatorial datum $\pi$. We denote by $\cT(\pi,I)$ the set of those $f\in\cG(\pi,I)$ which are IETs. Recall that the euclidian distance $d(x,y)$ between points $x,y$ in $\RR^2$ induces the \emph{Hausdorff distance} 
$\distance_H(E,F)$ between closed sets $E,F\subset [\closure(I)]^2$, that is 
$$
\distance_H(E,F):=
\max\left\{
\quad
\sup_{x\in E}\inf_{y\in F}d(x,y)
\quad,\quad
\sup_{x\in F}\inf_{y\in E}d(x,y)
\quad
\right\}.
$$
For $f\in\cG(\pi,I)$ and $\alpha\in\cA$ let $f_\alpha$ be its continuous branch as in Equation~\eqref{EquationPiecewiseHomeomorphism}, denote by $G_{f_\alpha}$ the graph of $f_\alpha$ and by 
$
\overline{G_{f_\alpha}}
$ 
the closure of $G_{f_\alpha}$ in $[\closure(I)]^2$. The set $\cG(\pi,I)$ is a metric space setting 
\begin{equation}
\label{EquationDinstanceGIETs}
\distance(f,f'):=
\sum_{\alpha\in\cA}
\distance_H(\overline{G_{f_\alpha}},\overline{G_{f'_\alpha}}).
\end{equation}

For any $\alpha\in\cA$ consider the points 
$
u^t_\alpha:=\inf I^t_\alpha
$ 
and 
$
u^b_\alpha:=\inf I^b_\alpha
$, 
so that in particular 
$
u_{\alpha_0}^t=u_{\beta_0}^b=\inf I
$ 
for the letters $\alpha_0,\beta_0$ with $\pi_t(\alpha_0)=\pi_b(\beta_0)=1$. The points $u_\alpha^t$ and $u_\alpha^b$ for $\alpha\in\cA$ are called respectively the \emph{critical points} and the \emph{critical values} of $f$. When $f$ is allowed to vary inside $\cG(\pi,I)$, we denote these points by 
$
u_\alpha^t(f)
$ 
and 
$
u_\alpha^b(f)
$. 
In particular, if $f=T(\pi,\lambda)$ is the IET determined by combinatorial-length data $(\pi,\lambda)$ then its critical values $u_\alpha^b(f)$ depend on $\lambda$ via the linear relation 
\begin{equation}
\label{EquationMarkingCriticalValues}
u_\alpha^b(f)=
\sum_{\pi_b(\chi)\leq\pi_b(\alpha)-1}\lambda_\chi
\quad
\textrm{ for any }
\quad
\alpha\in\cA.
\end{equation}
Similarly, the critical points of $f=T(\pi,\lambda)$ are 
$
u_\alpha^t(f)=\sum_{\pi_t(\chi)\leq\pi_t(\alpha)-1}\lambda_\chi
$ 
for $\alpha\in\cA$. Orbits of critical points and critical values contain relevant information on the dynamical properties of GIETs. In terms of such special orbits, Lemma~\ref{LemmaConnectionsRenormalizability} establishes a criterion which determines when an IET is \emph{infinitely renoramlizable}, or briefly a \emph{Keane IET} (see \S~\ref{SectionIterationRauzyInductionGIETs} for definitions). According to a well known result of M. Keane any such $T$ is also minimal (see \cite{Keane}, and also Lemma~\ref{LemmaConnectionsRenormalizability}). Moreover, according to \cite{MarmiMoussaYoccozCohomological}, the \emph{renormalization path} of any such $T$ is also \emph{infinite complete}, in the sense recalled in Definition~\ref{DefinitionRauzyPathOfGIET} below. More generally, in Definition~\ref{DefinitionRauzyPathOfGIET} it is introduced the notion of \emph{infinitely renormalizable} 
$
f\in\cG(\pi,I)
$ 
and the notion of \emph{Rauzy path} $\gamma(f,\infty)$ of any such GIET $f$. Differently from the case of classical IETs, in the general case infinite completeness must be added as an extra assumption. The notion of Rauzy path can be considered as a generalization of the \emph{rotation number} of a circle homeomorphism, and according to Proposition~\ref{PropositionPartitionDeterminesRauzyPath} such path is determined by the combinatorics of the orbits of critical points and critical values. The goal of this paper is to show that in a nice parameter families of GIETs, which in Definition~\ref{DefinitionFullFamily} are called \emph{full families}, one can find maps having any prescribed Rauzy path (see Theorem~\ref{TheoremFullFamily}). A motivation for this purpose is given by Proposition~\ref{PropositionSemiconjugation} below, which establishes an easy and standard result (see also Proposition 7 at page 45 in \cite{YoccozCollege}).

\begin{proposition}
[Poincar\'e, Yoccoz]
\label{PropositionSemiconjugation}
Let $T=T(\pi,\lambda)$ be any Keane IET determined by combinatorial-length data $(\pi,\lambda)$ and consider $f\in\cG(\pi,I)$, where 
$
I:=\big[0,\sum_{\chi\in\cA}\lambda_\chi\big)
$. 
If
$
\gamma(f,\infty)=\gamma(T,\infty)
$, 
that is $f$ and $T$ have the same Rauzy path, then $f$ is semi-conjugated to $T$, that is there exists a non decreasing, continuous and surjective map $h:I\to I$ such that
\begin{equation}
\label{EquationSemiconjugation}
h\circ f=T\circ h.
\end{equation}
\end{proposition}

Let $\Delta^\cA$ be the standard open simplex in $\RR^\cA$, that is the set of $\lambda\in\RR^\cA$ with $\lambda_\alpha>0$ for any $\alpha\in\cA$ and 
$
\sum_{\alpha\in\cA}\lambda_\alpha=1
$. 
If $T:I\to I$ is an IET defined on the interval $I:=[0,1)$, then its length datum belongs to 
$
\Delta^\cA
$, 
that is 
$
\cT(\pi,[0,1))=\{\pi\}\times\Delta^\cA
$. 
In the following, elements of $\Delta^\cA$ are used to parametrize more general families of GIETs via a marking of points in $[0,1)$ by the linear function in Equation~\eqref{EquationMarkingCriticalValues}, and in this general case they are denoted by the letter $\tau$, in order to avoid ambiguity with length data of IETs. Moreover such parameter families of GIETs will admit degenerations at the boundary of parameter space. We describe below the allowed degenerations.

\medskip

\begin{figure}
\begin{center}
{\begin{tikzpicture}[scale=0.06]



\draw[-] (0,0) -- (100,0) {};
\draw[-,thin,dashed] (-5,5) -- (65,5) {};
\draw[-,thin,dashed] (-5,35) -- (100,35) {};
\draw[-,thin,dashed] (-5,40) -- (100,40) {};
\draw[-,thin,dashed] (-5,70) -- (95,70) {};
\draw[-,thin,dashed] (-5,100) -- (60,100) {};

\draw[-] (0,0) -- (0,100) {};
\draw[-,thin,dashed] (20,-5) -- (20,100) {};
\draw[-,thin,dashed] (60,-5) -- (60,100) {};
\draw[-,thin,dashed] (65,-5) -- (65,70) {};
\draw[-,thin,dashed] (95,-5) -- (95,70) {};
\draw[-,thin,dashed] (100,-5) -- (100,40) {};


\node [circle,fill,inner sep=1.5pt] at (0,5) {};
\draw[-,thick] (0,5) .. controls (10,5) and (10,35) .. (20,35) {};
\node [circle,inner sep=1.5pt] at (20,35) {};
\node at (10,-5) {$A$};
\node at (-5,20) {$A$};


\node [circle,fill,inner sep=1.5pt] at (20,70) {};
\draw[-,thick] (20,70) .. controls (30,90) and (40,100) .. (60,100) {};
\node [circle,inner sep=1.5pt] at (60,100) {};
\node at (35,-5) {$B$};
\node at (-5,85) {$B$};


\node [circle,fill,inner sep=1.5pt] at (60,0) {};
\draw[-,thick] (60,0) .. controls (62,4) .. (65,5) {};
\node [circle,inner sep=1.5pt] at (65,5) {};
\node at (62,-5) {$E$};
\node at (-5,2) {$E$};


\node [circle,fill,inner sep=1.5pt] at (65,40) {};
\draw[-,thick] (65,40) .. controls (75,45) and (85,50) .. (95,70) {};
\node [circle,inner sep=1.5pt] at (65,70) {};
\node at (80,-5) {$C$};
\node at (-5,55) {$C$};


\node [circle,fill,inner sep=1.5pt] at (95,35) {};
\draw[-,thick] (95,35) -- (100,40) {};
\node [circle,inner sep=1.5pt] at (100,40) {};
\node at (97,-5) {$D$};
\node at (-5,37) {$D$};

\end{tikzpicture}}
\hspace{2 cm}
{\begin{tikzpicture}[scale=0.06]



\draw[-] (0,0) -- (100,0) {};
\draw[-,thin,dashed] (-5,35) -- (100,35) {};
\draw[-,thin,dashed] (-5,70) -- (100,70) {};
\draw[-,thin,dashed] (-5,100) -- (60,100) {};

\draw[-] (0,0) -- (0,100) {};
\draw[-,thin,dashed] (20,-5) -- (20,100) {};
\draw[-,thin,dashed] (65,-5) -- (65,100) {};
\draw[-,thin,dashed] (100,-5) -- (100,70) {};


\node [circle,fill,inner sep=1.5pt] at (0,0) {};
\draw[-,thick] (0,0) .. controls (10,5) and (10,35) .. (20,35) {};
\node [circle,inner sep=1.5pt] at (20,35) {};
\node at (10,-5) {$A$};
\node at (-5,20) {$A$};


\node [circle,fill,inner sep=1.5pt] at (20,70) {};
\draw[-,thick] (20,70) .. controls (30,90) and (40,100) .. (65,100) {};
\node [circle,inner sep=1.5pt] at (65,100) {};
\node at (40,-5) {$B$};
\node at (-5,85) {$B$};


\node [circle,fill,inner sep=1.5pt] at (65,0) {};


\node [circle,fill,inner sep=1.5pt] at (65,35) {};
\draw[-,thick] (65,35) .. controls (75,45) and (85,50) .. (100,70) {};
\node [circle,inner sep=1.5pt] at (65,70) {};
\node at (80,-5) {$C$};
\node at (-5,55) {$C$};


\node [circle,fill,inner sep=1.5pt] at (100,35) {};

\end{tikzpicture}}
\end{center}
\caption{In the left part of the figure it is represented a GIET $f$ with five intervals. On the right part of the figure a degeneration $D$ of such $f$. One can obtain $D$ from $f$ shrinking to zero both the pairs of intervals $I^{t/b}_D(f)$ and $I^{t/b}_E(f)$. This gives rise to two points, which are the singular part of $D$. The graph of $f$ is close to $D$ in the Hausdorff distance.}
\label{FigureDegenerationGIETs}
\end{figure}
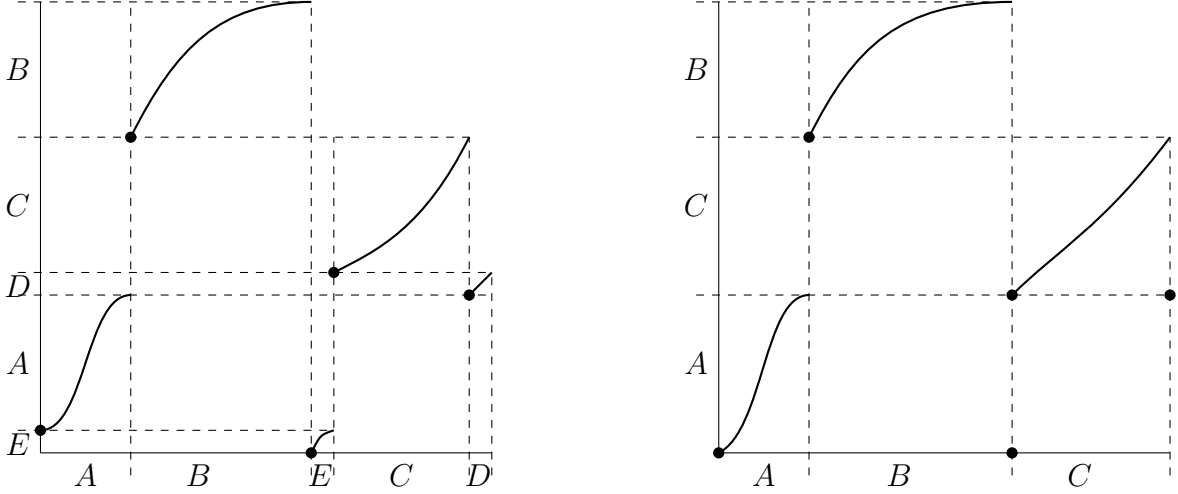

Fix a combinatorial datum $\pi$ over $\cA$. Let $\cA'\subset\cA$ be a sub-alphabet with $d'$ elements, where $1\leq d'\leq d-1$ and for 
$\epsilon\in\{t,b\}$ let 
$
\rho_\epsilon:\pi_\epsilon(\cA')\to\{1,\dots,d'\}
$ 
be the corresponding increasing bijections, which of course depend on $\cA'$. Let 
$\pi'=(\pi'_t,\pi'_b)$ be the combinatorial datum over $\cA'$ with 
$
\pi'_t,\pi'_b:\cA'\to\{1,\dots,d'\}
$ 
defined by 
$
\pi'_t:=\rho_t\circ\pi_t|_{\cA'}
$ 
and 
$
\pi'_b:=\rho_b\circ\pi_b|_{\cA'}
$. 
Such $\pi'$ is called a \emph{reduction} of $\pi$. In general $\pi'$ is not admissible even if $\pi$ is.

Let $I=[a,b)$ be a right-open interval. A \emph{degeneration} of a GIET in $\cG(\pi,I)$ is a subset 
$
D\subset[a,b]^2
$ 
which is the union $D=G_f\cup S$ of two pars. The \emph{regular part} $G_f$ of $D$ is the graph of a GIET $f\in\cG(\pi',I)$ whose combinatorial datum $\pi'$ is a reduction of $\pi$ with alphabet $\cA'\subset\cA$. The \emph{singular part} $S$ of $D$ is a finite set with $|\cA|-|\cA'|$ points, whose $x$-coordinate and $y$-coordinate are on the boundary of the continuity intervals of $f$ and $f^{-1}$ respectively, that is 
$$
S\subset
\big(\{u^t_\alpha(f),\alpha\in\cA'\}\cup\{b\}\big)
\times
\big(\{u^b_\alpha(f),\alpha\in\cA'\}\cup\{b\}\big).
$$
Let $\cD(\pi',I)$ be the set of degenerations $D$ whose regular part $f:I\to I$ has combinatorial datum $\pi'$. As it is shown in the right part of Figure~\ref{FigureDegenerationGIETs}, such degenerations can be obtained as limits, in the parameter $\tau\in\Delta^\cA$, of maps $f_\tau\in\cG(\pi,I)$, where for any $\alpha\in\cA'$ the branches $f_\tau|_{I^t_\alpha}$ in Equation~\eqref{EquationPiecewiseHomeomorphism} keep homeomorphism also in the limit, while the branches corresponding to letters in $\cA\setminus\cA'$ are contracted to a point in $S$. Finally consider the set  
$$
\widehat{\cG}(\pi,I):=
\cG(\pi,I)\cup
\bigcup_{\pi'\textrm{ reduction of }\pi}\cD(\pi',I).
$$ 
Consider $D\in\widehat{\cG}(\pi,I)$, decompose it as $D=G_f\cup S$ and let $\cA(D)\subset\cA$ be the alphabet of the regular part $f$ of $D$, where $S=\emptyset$ and $\cA(D)=\cA$ if 
$
D=f\in\cG(\pi,I)
$. 
When $S\not=\emptyset$ denote its points by $s_\alpha$, $\alpha\in\cA\setminus\cA(D)$. Then for any 
$\alpha\in\cA$ set 
$
D_\alpha:=\overline{G_{f_\alpha}}
$ 
if $\alpha\in\cA(D)$, and $D_\alpha:=\{s_\alpha\}$ otherwise. Finally extend to 
$\widehat{\cG}(\pi,I)$ the distance in Equation~\eqref{EquationDinstanceGIETs} setting 
\begin{equation}
\label{EquationDinstanceGIETsDegenerate}
\distance(D,D'):=
\sum_{\alpha\in\cA}\distance_H(D_\alpha,D'_\alpha).
\end{equation}

\begin{definition}
\label{DefinitionFullFamily}
A \emph{full family} over $\pi$ is a parameter family 
$
(f_\tau)_{\tau\in\Delta^\cA}
$ 
of GIETs which is the image of a map 
$$
F:\Delta^\cA\to\cG(\pi,[0,1))
\quad
,
\quad
\tau\mapsto F(\tau)=f_\tau
$$
satisfying the properties below.
\begin{enumerate}
\item
The map $F:\Delta^\cA\to\cG(\pi,[0,1))$ is continuous.
\item
For any $\tau\in\Delta^\cA$ the points defined by Equation~\eqref{EquationMarkingCriticalValues} are the critical values of $f_\tau$, that is for any $\alpha\in\cA$ we have
$$
u_\alpha^b(f_\tau)=
\sum_{\pi_b(\chi)\leq\pi_b(\alpha)-1}\tau_\chi.
$$
\item 
The map $F:\Delta^\cA\to\cG(\pi,[0,1))$ extends to a continuous map
$$
\widehat{F}:\overline{\Delta^\cA}\to\widehat{\cG}(\pi,[0,1)).
$$
\end{enumerate}
\end{definition}

Point (2) in Definition~\ref{DefinitionFullFamily} implies that the map $F$ must be injective. Moreover the linear functions  
$
\tau\mapsto \sum_{\pi_b(\chi)\leq\pi_b(\alpha)-1}\tau_\chi
$ 
defined by Equation~\eqref{EquationMarkingCriticalValues} extend continuously to the boundary of 
$\Delta^\cA$, and give a marking for the $y$-coordinate of points in the singular part $S$ of degenerations 
$
D\in\widehat{\cG}(\pi,[0,1))
$. 
If 
$
\tau^{(\infty)}\in\partial\Delta^\cA
$ 
and $\tau^{(n)}\in\Delta^\cA$ is a sequence with $\tau^{(n)}\to\tau^{(\infty)}$ as $n\to\infty$, then the continuity of $\widehat{F}$ implies that for any $\alpha\in\cA$ we must have
$$
\lim_{n\to\infty}u_\alpha^b(f_{\tau^{(n)}})=
\lim_{n\to\infty}
\sum_{\pi_b(\chi)\leq\pi_b(\alpha)-1}\tau^{(n)}_\chi
=
\sum_{\pi_b(\chi)\leq\pi_b(\alpha)-1}\tau^{(\infty)}_\chi,
$$
which is the analogous of Point (2) for the degeneration $D=\widehat{F}(\tau_\infty)$. If particular also 
$\widehat{F}$ must be injective. The map 
$
F:\Delta^\cA\to\cT(\pi,[0,1))
$, 
$
\lambda\mapsto f_\lambda:=T(\pi,\lambda)
$ 
in an example of a full family. In \S~\ref{SectionFullFamilies} we prove Proposition~\ref{PropositionFullFamilies} below.

\begin{proposition}
\label{PropositionFullFamilies}
Let $\pi$ be any combinatorial datum. There exists a continuous function 
$$
\cF:\cG\big(\pi,[0,1)\big)\times\Delta^\cA\to\cG\big(\pi,[0,1)\big)
\quad
\textrm{ , }
\quad
(f,\tau)\mapsto \cF(f,\tau)
$$
such that the following holds.
\begin{enumerate}
\item
For any $f\in\cG\big(\pi,[0,1)\big)$ we have a full family parametrized by 
$$
F:\Delta^\cA\to\cG\big(\pi,[0,1)\big)
\quad
\textrm{ , }
\quad
\tau\mapsto F(\tau):=\cF(f,\tau).
$$
\item
For any $f\in\cG\big(\pi,[0,1)\big)$ and any 
$
\tau,\tau'\in\Delta^\cA
$ 
we have
$$
\cF\big(\cF(f,\tau),\tau'\big)=\cF(f,\tau').
$$
In particular the space $\cG\big(\pi,[0,1)\big)$ is continuously foliated into full families.
\item
For any $f\in\cG\big(\pi,[0,1)\big)$, any $\tau\in\Delta^\cA$ and any $\alpha\in\cA$ the continuous branch $f_\alpha'$ of $f':=\cF(f,\tau)$ has the form 
$
f'_\alpha=\phi_\alpha\circ f_\alpha\circ \psi_\alpha^{-1}
$, 
where $f_\alpha$ is the corresponding continuous branch of $f$ and where $\phi_\alpha$ and $\psi_\alpha$ are affine functions. In particular $f_\alpha'$ and $f_\alpha$ have the same regularity.
\end{enumerate}
\end{proposition}

For any combinatorial datum $\pi$, Proposition~\ref{PropositionFullFamilies} provides natural parameter families of GIETs $(f_\tau)_{\tau\in\Delta^\cA}$ such that, for any Keane IET $T=T(\pi,\lambda)$, one can expect to find a parameter $\tau$ and a semi-conjugation $h$ such that $f_\tau$ and $T$ are related by Equation~\eqref{EquationSemiconjugation}. Unluckily, due to technical reasons appearing in \S~\ref{SectionProofTheoremRealizationRauzyPath} and in \S~\ref{SectionCyclicBehavior}, we can prove the existence of such $f_\tau$ only under an additional combinatorial assumption on the underlying combinatorial datum $\pi$, or more precisely on the \emph{Rauzy class} $\cR$ of $\pi$, which is defined in \S~\ref{SectionRauzyClasses}. Such combinatorial property seems to be satisfied in great generality, and in Remark~\ref{RemarkExistenceCyclicData} we discuss up to what extend it is in fact always true.

\begin{definition}
\label{DefinitionCyclicCombinatorialDatum}
Let $\pi$ be a combinatorial datum and   
$
\pi_t,\pi_b:\cA\to\{1,\dots,d\}
$ 
be the corresponding pair of bijections. Let 
$
\sigma=\sigma(\pi):\{1,\dots,d\}\to\{1,\dots,d\}
$ 
be the bijection given by 
$
\sigma:=\pi_b\circ\pi_t^{-1}
$. 
We say that $\pi$ is a \emph{cyclic} combinatorial datum if $\sigma(\pi)$ is cyclic of maximal order $d$.
\end{definition}

The main result of this paper is Theorem~\ref{TheoremFullFamily} below. We call it \emph{Full Family Theorem} for the analogy with similar results in the literature established for different kind of maps of the interval.

\begin{theorem}
\label{TheoremFullFamily}
Fix an admissible combinatorial datum $\pi$ and assume that the Rauzy class $\cR$ of $\pi$ contains a cyclic combinatorial datum $\pi^{(\ast)}$. Let $(f_\tau)_{\tau\in\Delta^\cA}$ be a full family over 
$\pi$. If $T\in\cT(\pi,[0,1))$ is a Keane IET, then there exists $\tau\in\Delta^\cA$ such that $f_\tau$ generates the same Rauzy path as $T$, that is 
$
\gamma(f,\infty)=\gamma(T,\infty)
$. 
\end{theorem}

In particular, if $f_\tau$ and $T$ are as in Theorem~\ref{TheoremFullFamily} above, then $f_\tau$ is semi-conjugated to $T$, according to Proposition~\ref{PropositionSemiconjugation}. Such $h$ is injective, and thus a conjugation, if and only if $f$ has no \emph{wandering intervals}. Since any Borel probability measure $\mu$ which is invariant under $f$ is not supported on wandering intervals, then any such $\mu$ is the pull-back $h^\ast\nu$ of a Borel probability measure $\nu$ invariant under $T$. In particular $f$ is uniquely ergodic if and only if $T$ is.

\begin{remark}
A criterion for the existence of affine interval exchange transformations with prescribed Rauzy path was previously established by Proposition 2.3 in~\cite{MarmiMoussaYoccozWandering}. 
\end{remark}

\begin{remark}
\label{RemarkExistenceCyclicData}
It is natural to ask if Theorem~\ref{TheoremFullFamily} holds for any admissible combinatorial datum $\pi$, and we don't know any Rauzy class for which the required property fails. It is easy to check directly that a cyclic $\pi^{(\ast)}$ exists in any Rauzy class over at most five intervals. The same is true for any $d\geq2$ in the \emph{hyperelliptic} Rauzy class $\cR^{hyp}_d$ over $d$ letters, that is the class of the datum $\pi$ with  
$
\pi_t(\alpha)+\pi_b(\alpha)=d+1
$ 
for any $\alpha\in\cA$, where the required $\pi^{(\ast)}$ can be obtained from $\pi$ alternating the \emph{top} and \emph{bottom} operations $R^t$ and $R^b$ (see \S~\ref{SectionRauzyClasses}) the right number of times. 
According to private discussions with V. Delecroix, R. Guti\'errez and A. Zorich, the existence of a cyclic $\pi^{(\ast)}$ can be proved also for other infinite lists of classes (see \cite{KontsevichZorich} for a classification of Rauzy classes). Finally, given any Rauzy class $\cR$ and any $\pi\in\cR$, one can get an other datum $\pi'$ with more intervals adding \emph{marked points} to $\pi$, which provides of course a more complex Rauzy class $\cR'$. According to \cite{DelecroixGoujardZografZorich}, this procedure eventually provides a class $\cR'$ which contains a cyclic element. This implies that for any $f\in\cG\big(\pi,[0,1)\big)$, allowing deformations with more parameters than in Definition~\ref{DefinitionFullFamily}, one can get $f_\tau$ satisfying the conclusion of Theorem~\ref{TheoremFullFamily}.
\end{remark}


\begin{remark}
Given a full family $F=(f_\tau)_{\tau\in\Delta}$ and a finite Rauzy path $\gamma$ one can consider the family $\widetilde{F}$ obtained from a proper subfamily $F_\gamma=(f_\tau)_{\tau\in\Delta_\gamma}$ of $F$ under the steps of the Rauzy induction specified by $\gamma$, where $\Delta_\gamma\subset\Delta$ is a proper open set. According to a private discussion with C. Fougeron and S. Ghazouani, it is possible to show that the family $\widetilde{F}$ is itself a full family, i.e. it satisfies properties (1), (2) and (3) in Definition~\ref{DefinitionFullFamily}. This was pointed out after this paper was finished, and provides the inductive step of a recursive alternative proof of our Theorem~\ref{TheoremRealizationRauzyPath}, and as a consequence of our main Theorem~\ref{TheoremFullFamily}.  Nevertheless this alternative argument becomes clear once the Definition~\ref{DefinitionFullFamily} is properly given, and such proper definition of full family was the final outcome of the attempt of implementing the \emph{Thurston map} (see \S~\ref{SectionThurstonMap}) in the setting of GIETs.
\end{remark}

\subsection{Contents of this paper}

The rest of the paper is organized as follows.

In \S~\ref{SectionFullFamilies} we prove Proposition~\ref{PropositionFullFamilies}. The proof is based on simple constructions, which are independent from the combinatorial structures introduced in \S~\ref{SectionRauzyClassesAndInductionGIETs} and \S~\ref{SectionThurstonMap}.

In \S~\ref{SectionRauzyClassesAndInductionGIETs} we recall the essential background on Rauzy induction. In particular, \S~\ref{SectionRauzyClasses} we introduce Rauzy classes, Rauzy paths and the corresponding linear co-cycle. In \S~\ref{SectionRauzyInductionGIETs} and \S~\ref{SectionIterationRauzyInductionGIETs} we explain how the induction applies to GIETs and finally in \S~\ref{SectionDynamicalPartition} we introduce dynamically defined partitions of the interval, which play a central role in all the paper.

In \S~\ref{SectionProofFullFamilyTheorem} we prove the Full Family Theorem~\ref{TheoremFullFamily}. The proof is given in \S~\ref{SectionProofTheoremFullFamily}, via Theorem~\ref{TheoremRealizationRauzyPath}, which is the main technical result in this paper. According to Theorem~\ref{TheoremRealizationRauzyPath} (and Remark~\ref{RemarkRealizationRauzyPath}), if the Rauzy class of $\pi$ contains a cyclic element, then in a full family over $\pi$ one can find a GIET generating a dynamical partition with given prescribed combinatorics. 

In \S~\ref{SectionThurstonMap} we introduce the \emph{Thurston map}, which is the main tool necessary to prove Theorem~\ref{TheoremRealizationRauzyPath}. In \S~\ref{SectionConfigurations} we fix a prescribed model for the combinatorics of a dynamical partition and we introduce the notion of \emph{configuration}, that is a partition of the interval into labelled subintervals respecting the combinatorics of the prescribed model, but not necessarily dynamically generated. The \emph{space of configurations} is naturally identified with an open simplex. Then in \S~\ref{SectionDefinitionThurstonMap} we define the \emph{Thurston map}, which acts continuously on the space of configurations, and in \S~\ref{SectionFixedPointsThurstonMap} we explain that fixed points of such map are dynamically defined partitions. 

In \S~\ref{SectionExistenceFixedPointThurstonMap} we complete the proof of Theorem~\ref{TheoremRealizationRauzyPath} by showing that the Thurston map has fixed points. The Brower fixed point Theorem cannot be applied on the open simplex, thus in \S~\ref{SectionBoundaryConfigurationSpace} we construct a boundary for it, introducing the notion of \emph{degenerate configuration}. In \S~\ref{SectionExtensionToTheBoundary} we define an extension of the Thurston map to the closed simplex, then in \S~\ref{SectionContinuityExtendedThurstonMap} we show its continuity. Finally, in \S~\ref{SectionCyclicBehavior} we show that the extended map has a cyclic behavior on the boundary of configuration space, so the fixed point must be in the interior. This last part uses the existence of a cyclic combinatorial datum.

\subsection{Acknowledgments*}

The authors would like to thank V. Delecroix, C. Fougeron, S. Ghazouani, R. Guti\'errez, C. Matheus, C. Tresser, C. Ulcigrai and A. Zorich for many  helpful discussions and for their advices. This work was supported by the Leverhulme Trust through the Leverhulme Prize of C. Ulcigrai and by the program Research in Paris (R.I.P.), sponsored by Institute Henri Poincar\'e. L. Marchese is grateful to French CNRS and De Giorgi Center for financial and logistic support during his stay in Pisa.

\section{Construction of full families}
\label{SectionFullFamilies}

Fix $f\in\cG\big(\pi,[0,1)\big)$ and for any $\beta\in\cA$ consider its corresponding critical value $u^b_\beta(f)$. Let also $I^t_\beta(f)$ and $I^b_\beta(f)$ be the right-open intervals in Equation~\eqref{EquationPiecewiseHomeomorphism}, so that the restriction 
$
f_\beta:I^t_\beta(f)\to I^b_\beta(f)
$ 
is an orientation preserving homeomorphism. In particular $u^b_\beta(f)$ is the left endpoint of $I^b_\beta(f)$. Fix also $\tau\in\Delta^\cA$ and for any $\beta\in\cA$ let 
$
\phi(f,\tau,\beta)>0
$ 
be the real number defined by
\begin{equation}
\label{EquationSlopesImage}
\phi(f,\tau,\beta):=
\frac{\tau_\beta}{u^b_{\beta^\ast}(f)-u^b_\beta(f)},
\end{equation}
where if $\pi_b(\beta)\leq d-1$ we denote by $\beta^\ast$ the letter with 
$
\pi_b(\beta^\ast)=\pi_b(\beta)+1
$, 
whereas if $\pi_b(\beta)=d$ we set $u^b_{\beta^\ast}(f):=1$. Let 
$
\phi_{(f,\tau)}:[0,1)\to[0,1)
$ 
be the unique piecewise affine map whose slope on each $I^b_\beta(f)$ is $\phi(f,\tau,\beta)$, that is 
$$
\phi_{(f,\tau)}(x):=
\int_0^x \sum_{\beta\in\cA}\unit^b_{f,\beta}(y)\cdot\phi(f,\tau,\beta)dy,
$$
where $\unit^b_{f,\beta}(\cdot)$ denotes the characteristic function of $I^b_\beta(f)$ that is the function defined by 
$
\unit^t_{f,\beta}(y):=1
$ 
if $y\in I^b_\beta(f)$ and $\unit^b_{f,\beta}(y):=0$ if 
$
y\in [0,1)\setminus I^b_\beta(f)
$. 
The map $\phi_{(f,\tau)}$ is obviously a continuous increasing bijection, and for any $\beta\in\cA$ we have
\begin{equation}
\label{EquationSlopesImage(2)}
\phi_{(f,\tau)}\big(u^b_\beta(f)\big)
=
\sum_{\pi_b(\chi)\leq\pi_b(\beta)-1}\tau_\chi.
\end{equation}
Observe that
$
\sum_{\beta\in\cA}\phi(f,\tau,\beta)\cdot|I^b_\chi(f)|=\sum_{\beta\in\cA}\tau_\beta=1
$, 
then define the \emph{rescaling factor}
\begin{equation}
\label{EquationHorizontalRescaling}
\lambda(f,\tau):=\sum_{\beta\in\cA}\phi(f,\tau,\beta)\cdot|I^t_\chi(f)|.
\end{equation}
We have $\phi(f,\tau,\beta)\geq1$ for at least one $\beta\in\cA$. On the other hand, since    
$
\sum_{\beta\in\cA}|I^t_\beta(f)|=1
$, 
then 
$
\phi(f,\tau,\beta)<|I^t_\beta(f)|^{-1}
$ 
for any $\beta$. Thus we have  
\begin{equation}
\label{EquationBoundsHorizontalRescaling}
\min_{\beta\in\cA}|I^t_\beta(f)|
<
\lambda(f,\tau)
<
\sum_{\beta\in\cA}\frac{1}{|I^t_\beta(f)|}.
\end{equation}
For any $\alpha\in\cA$ let 
$
\psi(f,\tau,\alpha)>0
$ 
be the real number defined by 
\begin{equation}
\label{EquationSlopesDomain}
\psi(f,\tau,\alpha):=\frac{\phi(f,\tau,\alpha)}{\lambda(f,\tau)},
\end{equation}
then let 
$
\psi_{(f,\tau)}:[0,1)\to[0,1)
$ 
be the unique piecewise affine map whose slope on any $I^t_\alpha(f)$ is $\psi(f,\tau,\alpha)$, that is 
$$
\psi_{(f,\tau)}(x):=
\int_0^x \sum_{\alpha\in\cA}\unit^t_{f,\alpha}(y)\cdot\psi(f,\tau,\alpha)dy,
$$
where $\unit^t_{f,\alpha}(\cdot)$ is the characteristic function of $I^t_\alpha(f)$ for any 
$\alpha\in\cA$. Obviously $\psi_{(f,\tau)}$ is a continuous increasing bijection, and the same is true for its inverse $\psi_{(f,\tau)}^{-1}$. Moreover we have 
$$
\psi_{(f,\tau)}\big(u^t_\alpha(f)\big)
=
\sum_{\pi_t(\chi)\leq\pi_t(\alpha)-1}\tau_\chi
\quad
\textrm{ for any }
\quad
\alpha\in\cA.
$$
Finally, consider the map 
\begin{equation}
\label{EquationConstructionFullFamily}
\cF:\cG\big(\pi,[0,1)\big)\times\Delta^\cA\to\cG\big(\pi,[0,1)\big)
\quad
\textrm{ , }
\quad
(f,\tau)\mapsto \cF(f,\tau):=\phi_{(f,\tau)}\circ f\circ \psi_{(f,\tau)}^{-1}.
\end{equation}

\subsection{Proof of Proposition~\ref{PropositionFullFamilies}}

In this subsection we prove that the function $\cF$ defined by Equation~\eqref{EquationConstructionFullFamily} satisfies Point (1), Point (2) and Point (3) in Proposition~\ref{PropositionFullFamilies}. Let us observe that Point (3) is evident from the definition of $\cF$, thus we just prove the first two.

\medskip

\emph{Proof of Point (2)}. Fix $f\in\cG\big(\pi,[0,1)\big)$ and 
$
\tau,\tau'\in\Delta^\cA
$, 
then set $f':=\cF(f,\tau)$. For $\cF$ defined by Equation~\eqref{EquationConstructionFullFamily}, the statement in Point (1) is equivalent to
$$
\phi_{(f',\tau')}\circ\phi_{(f,\tau)}
\circ f\circ 
\psi_{(f,\tau)}^{-1}\circ\phi_{(f',\tau')}^{-1}
=
\phi_{(f,\tau')}\circ f\circ \psi_{(f,\tau')}^{-1}
$$
thus it is enough to prove that we have 
$$
\phi_{(f,\tau')}=\phi_{(f',\tau')}\circ\phi_{(f,\tau)}
\quad
\textrm{ and }
\quad
\psi_{(f,\tau')}=\psi_{(f',\tau')}\circ\psi_{(f,\tau')}.
$$
All the functions in the identities above are continuous piecewise affine bijections of $[0,1)$ thus the identities are satisfied if and only if the corresponding relations for derivatives hold. This last property is easy to check, indeed Equation~\eqref{EquationSlopesImage(2)} implies that for any $\beta$ we have 
$$
u^b_{\beta^\ast}(f')-u^b_\beta(f')=
\phi_{(f,\tau)}\big(u^b_{\beta^\ast}(f)\big)-
\phi_{(f,\tau)}\big(u^b_{\beta}(f)\big)
=
\tau_\beta,
$$ 
so that Equation~\eqref{EquationSlopesImage} implies
$$
\phi(f,\tau',\beta)=
\frac{\tau'_\beta}{u^b_{\beta^\ast}(f)-u^b_\beta(f)}=
\frac{\tau'_\beta}{u^b_{\beta^\ast}(f')-u^b_\beta(f')}\cdot
\frac{\tau_\beta}{u^b_{\beta^\ast}(f)-u^b_\beta(f)}=
\phi(f',\tau',\beta)\cdot\phi(f,\tau,\beta).
$$
Moreover observing that 
$
|I^t_\chi(f')|=|I^t_\chi(f)|\cdot \psi(f,\tau,\chi)
$ 
for any $\chi\in\cA$, then the relation above and Equation~\eqref{EquationSlopesDomain} give 
\begin{align*}
&
\lambda(f,\tau')
=
\sum_{\chi\in\cA}\phi(f,\tau',\chi)\cdot|I^t_\chi(f)|
=
\sum_{\chi\in\cA}
\frac{\phi(f,\tau',\chi)}{\psi(f,\tau,\chi)}\cdot|I^t_\chi(f')|=
\\
&
\sum_{\chi\in\cA}
\frac{\phi(f,\tau,\chi)}{\psi(f,\tau,\chi)}\cdot\phi(f',\tau',\chi)\cdot|I^t_\chi(f')|
=
\lambda(f,\tau)\cdot\sum_{\chi\in\cA}\phi(f',\tau',\chi)\cdot|I^t_\chi(f')|
=
\lambda(f,\tau)\cdot\lambda(f',\tau').
\end{align*}
It follows that for any $\beta\in\cA$ we have also
$$
\psi(f,\tau',\beta)
=
\frac{\phi(f,\tau',\beta)}{\lambda(f,\tau')}
=
\frac{\phi(f,\tau,\beta)}{\lambda(f,\tau)}
\cdot
\frac{\phi(f',\tau',\beta)}{\lambda(f',\tau')}
=
\psi(f',\tau',\beta)\cdot\psi(f,\tau,\beta).
$$

\medskip

\emph{Proof of Point (1)}. We prove Points (1), (2) and (3) in Definition~\ref{DefinitionFullFamily}. Point (2) is a direct consequence of the definition of the map $\phi_{(f,\tau)}$. Point (1) is a consequence of the continuity of the map $\cF$ in Equation~\eqref{EquationConstructionFullFamily}, which follows from the continuity of the maps 
\begin{align}
&
\label{Equation(1)ProofFullFamily}
\cG\big(\pi,[0,1)\big)\times\Delta^\cA\to C^0\big([0,1)\big)
\quad,\quad
(f,\tau)\mapsto \phi_{(f,\tau)}
\\
&
\label{Equation(2)ProofFullFamily}
\cG\big(\pi,[0,1)\big)\times\Delta^\cA\to (0,+\infty)
\quad,\quad
(f,\tau)\mapsto \lambda(f,\tau),
\end{align}
where $C^0\big([0,1)\big)$ denotes the set of continuous maps of the closed interval $[0,1)$ with the sup norm $\|\cdot\|_\infty$. We finish the proof of Proposition~\ref{PropositionFullFamilies} showing that Point (3) in Definition~\ref{DefinitionFullFamily} is satisfied. Fix 
$
f\in\cG\big(\pi,[0,1)\big)
$ 
and $\tau\in\partial\Delta^\cA$. Let $\cZ(\tau)\subset\cA$ be the non-empty subset of those 
$\alpha\in\cA$ such that 
$
\tau_\alpha=0
$. 
Equation~\eqref{EquationSlopesImage} and Equation~\eqref{EquationSlopesDomain} extend continuously to 
$
\tau\in\partial\Delta^\cA
$, 
defining quantities $\phi(f,\tau,\alpha)$ and $\psi(f,\tau,\alpha)$. Observe that 
$$
\phi(f,\tau,\alpha)=0
\quad
\Leftrightarrow
\quad
\tau_\alpha=0
\quad
\Leftrightarrow
\quad
\psi(f,\tau,\alpha)=0
\quad
\textrm{ for any }
\quad
\alpha\in\cA,
$$
where the first equivalence is obvious, while the second holds because $\lambda(f,\tau)$ is bounded from above and from below, uniformly in $\tau$, according to 
Equation~\eqref{EquationBoundsHorizontalRescaling}. Since $\cZ(\tau)\not=\emptyset$, then the continuous piecewise affine maps $\phi_{(f,\tau)}$ and $\psi_{(f,\tau)}$ are still surjective, but not injective. Nevertheless, removing the intervals $I^{t/b}_\alpha(f)$, $\alpha\in\cZ(\tau)$, they still define bijections
\begin{align*}
&
\phi_{(f,\tau)}:
\bigsqcup_{\beta\in\cA\setminus\cZ(\tau)}I^b_\beta(f)\to[0,1)
\quad\quad
\textrm{ and }
\quad\quad
\psi_{(f,\tau)}:
\bigsqcup_{\alpha\in\cA\setminus\cZ(\tau)}I^t_\alpha(f)\to[0,1).
\end{align*}
Hence the composition satisfies 
$
\phi_{(f,\tau)}\circ f\circ \psi_{(f,\tau)}^{-1}
\in
\cG\big(\pi',[0,1)\big)
$, 
where $\pi'$ is the combinatorial datum obtained removing the letters in $\cZ(\tau)$ from both lines of $\pi$. For any $\tau\in\partial\Delta^\cA$ as above define a degeneration $D=\widehat{F}(\tau)$ setting 
$$
\widehat{F}(\tau):=G_{f_\tau}\cup S_\tau
\in
\widehat{\cG}\big(\pi,[0,1)\big),
$$
where 
$
f_\tau:=\phi_{(f,\tau)}\circ f\circ \psi_{(f,\tau)}^{-1}
$ 
is the GIET given by the composition above and $G_{f_\tau}\subset[0,1)^2$ is its graph, and where the singular part 
$
S_\tau=\{s_\alpha\}_{\alpha\in\cZ(\tau)}
$ 
is the family of points $s_\alpha\in[0,1]^2$ given by
$$
s_\alpha:=
\bigg(
\sum_{\pi_t(\chi)\leq\pi_t(\alpha)}\tau_\chi,\sum_{\pi_b(\chi)\leq\pi_b(\alpha)}\tau_\chi
\bigg)
\quad
\textrm{ for any }
\quad
\alpha\in\cZ(\tau).
$$
Setting $\widehat{F}(\tau):=\cF(f,\tau)$ for any non degenerate $\tau\in\Delta^\cA$, we get the required extension map 
$
\widehat{F}:\overline{\Delta^\cA}\to\widehat{\cG}\big(\pi,[0,1)\big)
$. 
Such $\widehat{F}$ is continuos because the maps defined by Equation~\eqref{Equation(1)ProofFullFamily} and Equation~\eqref{Equation(2)ProofFullFamily} extend continuously to 
$
\cG\big(\pi,[0,1)\big)\times\overline{\Delta^\cA}
$. 
Proposition~\ref{PropositionFullFamilies} is proved. $\qed$

\section{Rauzy induction on GIETs}
\label{SectionRauzyClassesAndInductionGIETs}

In this section we recall some background on the \emph{Rauzy induction} on GIETs (see \cite{Rauzy}), which is also known as \emph{Rauzy-Veech induction} for his relation with the \emph{Teichm\"uller flow} on the moduli space of \emph{translation surfaces} (see \cite{Veech}). We also follow \cite{MarmiMoussaYoccozCohomological}.

\subsection{Rauzy classes and Rauzy matrices}
\label{SectionRauzyClasses}

Let $\cA$ be a finite alphabet with $d\geq2$ letters. We define two operations $R^t$ and $R^b$ on the set of admissible combinatorial data $\pi$ over $\cA$, where the symbols "$t$" and "$b$" stand for \emph{top} and \emph{bottom} respectively. It is practical to introduce the variable 
$
\epsilon\in\{t,b\}
$. 
Fix an admissible combinatorial datum $\pi=(\pi_t,\pi_b)$ and let $\alpha_t$ and $\alpha_b$ be the letters in $\cA$ such that 
$
\pi_t(\alpha_t)=\pi_b(\alpha_b)=d
$. 
Below, for $\epsilon\in\{t,b\}$, we describe the combinatorial datum  
$
\widetilde{\pi}=(\widetilde{\pi}_t,\widetilde{\pi}_b)
$, 
where 
$
\widetilde{\pi}:=R^\epsilon(\pi)
$.

\begin{description}
\item
[Top operation]
The letter $\alpha_t$ is said the \emph{winner} of the top operation $R^t$ and $\alpha_b$ is said the \emph{looser}. The top operation 
$
\widetilde{\pi}:=R^t(\pi)
$ 
leaves invariant $\pi_t$, that is 
$
\widetilde{\pi}_t:=\pi_t
$, 
and its action on $\pi_b$ is defined by
\begin{eqnarray*}
&&
\widetilde{\pi}_b(\chi):=\pi_b(\chi)
\quad
\textrm{ for }
\quad
1\leq\pi_b(\chi)\leq\pi_b(\alpha_t)
\\
&&
\widetilde{\pi}_b(\alpha_b):=\pi_b(\alpha_t)+1
\\
&&
\widetilde{\pi}_b(\chi):=\pi_b(\chi)+1
\quad
\textrm{ for }
\quad
\pi_b(\alpha_t)+1\leq\pi_b(\chi)\leq d-1.
\end{eqnarray*}
\item
[Bottom operation]
The letter $\alpha_b$ is said the \emph{winner} of the bottom operation $R^b$ and $\alpha_t$ is said the \emph{looser}. The top operation 
$
\widetilde{\pi}:=R^b(\pi)
$ 
leaves invariant $\pi_b$, that is 
$
\widetilde{\pi}_b:=\pi_b
$, 
and its action on $\pi_t$ is defined by
\begin{eqnarray*}
&&
\widetilde{\pi}_t(\chi):=\pi_t(\chi)
\quad
\textrm{ for }
\quad
1\leq\pi_t(\chi)\leq\pi_t(\alpha_b)
\\
&&
\widetilde{\pi}_t(\alpha_t):=\pi_t(\alpha_b)+1
\\
&&
\widetilde{\pi}_t(\chi):=\pi_t(\chi)+1
\quad
\textrm{ for }
\quad
\pi_t(\alpha_b)+1\leq\pi_t(\chi)\leq d-1.
\end{eqnarray*}
\end{description}

It is easy to check that both $R^b(\pi)$ and $R^t(\pi)$ are admissible if $\pi$ is. A \emph{Rauzy class} $\cR$ is a set of admissible combinatorial data which is invariant both under $R^t$ and $R^b$ and which is minimal with such property. The \emph{Rauzy diagram} $\cD$ is the connected oriented graph whose vertexes are the elements of $\cR$ and whose elementary oriented arcs $\gamma$, or \emph{arrows}, correspond to Rauzy elementary operations. The set of arrows $\gamma$ of $\cD$ is in bijection with the set of pairs $(\pi,\epsilon)$ with $\pi\in\cR$ and $\epsilon\in\{t,b\}$. A concatenation of $r$ compatible arrows $\gamma_1,\dots,\gamma_r$ in a Rauzy diagram is called a \emph{Rauzy path} and is denoted $\gamma=\gamma_1\ast\dots\ast\gamma_r$. If $\pi,\pi'$ in $\cR$ are respectively the initial and the final combinatorial data in such chain of combinatorial operations, we write also $\gamma:\pi\to\pi'$. If a path $\gamma$ is concatenation of $r$ simple arrows, we say that $\gamma$ has length $r$. Length one paths are arrows. Elements of $\cR$ are identified with trivial length-zero paths. Let 
$
\{e_\chi\}_{\chi\in\cA}
$ 
be the canonical basis of $\RR^\cA$. For any finite path $\gamma$ define a linear map $B_\gamma\in\slgroup(d,\ZZ)$ as follows. If $\gamma$ is trivial then $B_\gamma:=id$. If $\gamma$ is an arrow with winner $\alpha$ and loser $\beta$ set
$$
B_\gamma e_\alpha=e_\alpha+e_\beta
\quad
\textrm{ and }
\quad
B_\gamma e_\chi=e_\chi
\textrm{ for }
\chi\not=\alpha.
$$
Then extend the definition to paths so that for any concatenation $\gamma_1\ast\gamma_2$ we have 
$$
B_{\gamma_1\ast\gamma_2}=B_{\gamma_2}\cdot B_{\gamma_1}.
$$

\subsection{The Rauzy induction map}
\label{SectionRauzyInductionGIETs}

Fix an admissible combinatorial datum $\pi$ over $\cA$, a right-open interval $I$ and a map 
$f\in\cG(\pi,I)$. Consider the two corresponding partitions 
$
\cP_t=\{I^t_\alpha\}_{\alpha\in\cA}
$ 
and 
$
\cP_b=\{I^b_\alpha\}_{\alpha\in\cA}
$ 
of $I$. It is practical to keep track of the dependence on $f$ of such partitions and their atoms, so we write
\begin{equation}
\label{EquationDynamicalPartition}
\cP_t(f)=\{I^t_\alpha(f)\}_{\alpha\in\cA}
\quad
\textrm{ and }
\quad
\cP_b(f)=\{I^b_\alpha(f)\}_{\alpha\in\cA}.
\end{equation}
For $\alpha\in\cA$, consider the critical points 
$
u_\alpha^t=u_\alpha^t(f)
$ 
of $f$ and the critical values 
$
u_\alpha^b=u_\alpha^b(f)
$. 
In particular, observe that if $f=T(\pi,\lambda)$ is the IET defined by the  combinatorial-length data $(\pi,\lambda)$ and acting on the interval $I=[0,\sum_\chi\lambda_\chi)$, then we have
$$
u_\alpha^t(f)=
\sum_{\pi_t(\chi)\leq\pi_t(\alpha)-1}\lambda_\chi
\quad
\textrm{ and }
\quad
u_\alpha^b(f)=
\sum_{\pi_b(\chi)\leq\pi_b(\alpha)-1}\lambda_\chi.
$$
As in \S~\ref{SectionRauzyClasses}, let $\alpha_t$ and $\alpha_b$ be the letters with $\pi_t(\alpha_t)=\pi_b(\alpha_b)=d$. Suppose that the following condition is satisfied
\begin{equation}
\label{EquationConditionRenormalizableGIET}
u_{\alpha_t}^t(f)\not=u_{\alpha_b}^b(f).
\end{equation}
Assign a value to the variable 
$
\epsilon=\epsilon(f)\in\{t,b\}
$ 
according to the two cases below: 
\begin{eqnarray*}
&&
\epsilon(f):=t
\quad
\Leftrightarrow
\quad
u_{\alpha_t}^t(f)<u_{\alpha_b}^b(f)
\\
&&
\epsilon(f):=b
\quad
\Leftrightarrow
\quad
u_{\alpha_t}^t(f)>u_{\alpha_b}^b(f),
\end{eqnarray*}
In particular for $f=T(\pi,\lambda)$ we have 
$
\epsilon(f)=t\Leftrightarrow\lambda_{\alpha_t}>\lambda_{\alpha_b}
$ 
and 
$
\epsilon(f)=b\Leftrightarrow\lambda_{\alpha_t}<\lambda_{\alpha_b}
$. 

\medskip

Let $\widetilde{\pi}:=R^{\epsilon(f)}(\pi)$ be the combinatorial datum and 
$
\gamma:\pi\to\widetilde{\pi}
$ 
be the arrow in the Rauzy diagram $\cD$ corresponding to the pair 
$
(\pi,\epsilon=\epsilon(f))
$ 
as in \S~\ref{SectionRauzyClasses}. Let also $B_\gamma\in\slgroup(d,\ZZ)$ be the matrix associated to $\gamma$ as in \S~\ref{SectionRauzyClasses}. Let $\widetilde{I}\subset I$ be the subinterval defined by 
$$
\inf\widetilde{I}=\inf I
\quad
\textrm{ and }
\quad
\sup\widetilde{I}:=\max\{u_{\alpha_t}^t(f),u_{\alpha_b}^b(f)\}.
$$

Let $\widetilde{f}:\widetilde{I}\to\widetilde{I}$ be the first return map of $f$ to $\widetilde{I}$. Recall that for any $\alpha\in\cA$ we denote by $f_\alpha$ the corresponding continuous branch defined by Equation~\eqref{EquationPiecewiseHomeomorphism}. The explicit expression for $\widetilde{f}$ is given according to the two cases below. 
\begin{enumerate}
\item
If $\epsilon(f)=t$ then, observing that 
$
I^b_{\alpha_b}=f_{\alpha_b}(I^t_{\alpha_b})\subset I^t_{\alpha_t}
$, 
we have
\begin{align*}
&
\widetilde{f}(x)=f^2(x)=f_{\alpha_t}\circ f_{\alpha_b}(x)
&
\quad
\textrm{ for any }
\quad
x\in I_{\alpha_b}^t
\\
&
\widetilde{f}(x)=f(x)
&
\quad
\textrm{ for any }
\quad
x\in\widetilde{I}\setminus I_{\alpha_b}^t.
\end{align*}
\item
If $\epsilon(f)=b$ then, observing that  
$
I_{\alpha_t}^t\subset I_{\alpha_b}^b=f_{\alpha_b}(I_{\alpha_b})
$, 
we have
\begin{align*}
&
\widetilde{f}(x)=f^2(x)=f_{\alpha_t}\circ f_{\alpha_b}(x)
&
\quad
\textrm{ for any }
\quad
x\in f_{\alpha_b}^{-1}(I_{\alpha_t}^t) 
\\
&
\widetilde{f}(x)=f(x)
&
\quad
\textrm{ for any }
\quad
x\in\widetilde{I}\setminus f_{\alpha_b}^{-1}(I_{\alpha_t}^t).
\end{align*}
\end{enumerate}

It is easy to see that $\widetilde{f}$ is a GIET and more precisely  
$
\widetilde{f}\in\cG(\widetilde{\pi},\widetilde{I})
$, 
where 
$
\widetilde{\pi}=R^{\epsilon(f)}(\pi)
$, 
so that we have a map
$$
\cQ:f\mapsto \cQ(f):=\widetilde{f},
$$
called \emph{Rauzy map}. The map $\cQ$ preserves the set of IETs. In particular for $f=T(\pi,\lambda)$ we have 
$
\widetilde{f}=T(\widetilde{\pi},\widetilde{\lambda})
$,  
where the new length datum is given by 
$
\widetilde{\lambda}=^t\!B_\gamma^{-1}\lambda
$. 
Explicitly, setting $1-\epsilon:=b$ if $\epsilon=t$ and otherwise $1-\epsilon:=t$ if $\epsilon=b$, we have
\begin{eqnarray*}
&&
\widetilde{\lambda}_{\alpha}=\lambda_{\alpha}
\quad
\textrm{ if }
\quad
\alpha\neq\alpha_{\epsilon}
\\
&&
\widetilde{\lambda}_{\alpha_{\epsilon}}=
\lambda_{\alpha_{\epsilon}}-\lambda_{\alpha_{1-\epsilon}}.
\end{eqnarray*}

\subsection{Iteration of the Rauzy induction map}
\label{SectionIterationRauzyInductionGIETs}

Let $\cR$ be a Rauzy class over $\cA$ and consider $\pi\in\cR$. Fix $f\in\cG(\pi,I)$ and $r\in\NN$ and assume that the $r$-th iterated $f^{(r)}:=\cQ^r(f)$ of the Rauzy map is defined on $f=f^{(0)}$. Let $\pi^{(r)}\in\cR$ be the combinatorial datum of $f^{(r)}$ and $I^{(r)}$ be the interval where $f^{(r)}$ acts, so that in our notation we write 
$
f^{(r)}\in\cG(\pi^{(r)},I^{(r)})
$, 
which means that the map $f^{(r)}:I^{(r)}\to I^{(r)}$ is a GIET with combinatorial datum $\pi^{(r)}$. Finally let $\alpha^{(r)}_t$ and $\alpha^{(r)}_b$ be the letters such that 
$
\pi^{(r)}_t(\alpha^{(r)}_t)=\pi^{(r)}_b(\alpha^{(r)}_b)=d
$ 
and assume that $f^{(r)}$ satisfies Condition~\eqref{EquationConditionRenormalizableGIET}, that is 
$$
u^t_{\alpha^{(r)}_t}(f^{(r)})\not=u^b_{\alpha^{(r)}_b}(f^{(r)}).
$$
Then the interval $I^{(r+1)}\subset I^{(r)}$ and the map 
$
f^{(r+1)}=\cQ(f^{(r)})\in\cG(\pi^{(r+1)},I^{(r+1)})
$ 
are defined inductively as in \S~\ref{SectionRauzyInductionGIETs}, that is
\begin{eqnarray*}
&&
\inf I^{(r+1)}=\inf I^{(r)}
\quad
\textrm{ and }
\quad
\sup I^{(r+1)}:=
\max\{u^t_{\alpha^{(r)}_t}(f^{(r)}),u^b_{\alpha^{(r)}_b}(f^{(r)})\}
\\
&&
\pi^{(r+1)}:=R^{\epsilon(r)}(\pi^{(r)})
\quad
\textrm{ where }
\quad
\epsilon(r)=\epsilon(f^{(r)}).
\end{eqnarray*}
Moreover the operation above is encoded by the arrow 
$
\gamma_{r+1}:\pi^{(r)}\to\pi^{(r+1)}
$ 
in the Rauzy diagram corresponding to the data 
$
(\pi^{(r)},\epsilon(f^{(r)}))
$.

\begin{definition}
\label{DefinitionRauzyPathOfGIET}
Consider $f\in\cG(\pi,I)$.
\begin{enumerate}
\item
The map $f$ is said \emph{infinitely renormalizable} if $f^{(r)}$ satisfies Condition~\eqref{EquationConditionRenormalizableGIET} for any $r\in\NN$ and moreover, for any $\alpha\in\cA$, there exists infinitely many $r$ such that $\alpha$ is the winner of the arrow $\gamma_r$.
\item
The \emph{Rauzy path} 
$
\gamma(f,r):\pi^{(0)}\to\pi^{(r)}
$ 
of $f$ up to time $r$ is the path in the Rauzy diagram $\cD$ obtained by concatenation of the arrows above, that is
$$
\gamma(f,r):=
\gamma_1\ast\dots\ast\gamma_r.
$$
\item
The infinite Rauzy path of $f$ is the infinite concatenation 
$
\gamma(f,+\infty)=\gamma_1\ast\gamma_2\ast\dots
$. 
According to Point (1) above, we say that such path is \emph{infinite complete}. 
\end{enumerate}
\end{definition}

Observe that 
$
B_{\gamma_1\ast\dots\ast\gamma_r}=B_{\gamma_r}\cdot\dots\cdot B_{\gamma_1}
$ 
according to its definition in \S~\ref{SectionRauzyClasses}. Therefore, if $f=T(\pi,\lambda)$ is the IET defined by combinatorial-length data $(\pi,\lambda)$, then according to \S~\ref{SectionRauzyInductionGIETs} we have 
$
f^{(r)}=T(\pi^{(r)},\lambda^{(r)})
$ 
for any $r\in\NN$, where the $r$-th length datum is given by
\begin{equation}
\label{EquationLinearActionLengthData}
\lambda^{(r)}=^t\!B_{\gamma_1\ast\dots\ast\gamma_r}^{-1}(\lambda).
\end{equation}
The space of IETs with combinatorial datum in $\cR$ is identified with 
$
\cR\times\RR_+^\cA
$ 
and according to Equation~\eqref{EquationLinearActionLengthData} the Rauzy map $\cQ$ acts as a piecewise linear map on this space. More precisely, fix $r\in\NN$ and let 
$
\gamma=\gamma_1\ast\dots\ast\gamma_r
$ 
be the concatenation of $r$ Rauzy arrows, where 
$
\gamma_k:\pi^{(k-1)}\to\pi^{(k)}
$ 
for any $k=1,\dots,r$, then set $\pi:=\pi^{(0)}$ and $\pi':=\pi^{(r)}$, so that $\gamma:\pi\to\pi'$. Observe that all entries of $B_\gamma$ are non-negative and define the simplicial sub-cone 
$
\cC_\gamma:=^t\!B_\gamma(\RR_{+}^{\cA})\subset\RR^\cA
$.  
The $r$-th iterate $\cQ^r$ is defined on any 
$
f=T(\pi,\lambda)\in\{\pi\}\times\cC_\gamma
$ 
and its restriction to such set coincides is given by the linear map
$$
\cQ^r:\{\pi\}\times\cC_\gamma \to \{\pi'\}\times\RR^\cA_+ 
\quad
\textrm{ , }
\quad
f=T(\pi,\lambda)\mapsto f^{(r)}=T(\pi',^t\!B_\gamma^{-1}\lambda).
$$
The set of IETs $f=T(\pi,\lambda)$ where $\cQ^r$ is not defined is the union of the sets  
$
\{\pi\}\times \partial\cC_\gamma
$ 
over all $\pi\in\cR$ and all Rauzy paths $\gamma:\pi\to\pi'$ of length $r$, where $\partial\cC_\gamma$ denotes the boundary of $\cC_\gamma$. Therefore infinitely renormalizable IETs form a set with full Lebesgue measure, its complement being contained into a countable union of hyperplanes. The following combinatorial characterization holds (for a proof see Corollary 4 at page 40 in  \cite{YoccozCollege}).

\begin{lemma}
\label{LemmaConnectionsRenormalizability}
Let $T=T(\pi,\lambda)$ be the IET defined by combinatorial-length data $(\pi,\lambda)$. Then $T$ is infinitely renormalizable if any only if it does
not have \emph{connections}, that is triples $(\beta,\alpha,n)$ with $n\geq 0$ and $\pi_b(\beta)\geq2$, $\pi_t(\alpha)\geq2$ such that 
$$
T^n(u_\beta^b)=u_\alpha^t.
$$
\end{lemma}

\subsection{Dynamically defined partitions}
\label{SectionDynamicalPartition}

Fix $\pi\in\cR$ and let 
$
\gamma=\gamma_1\ast\dots\ast\gamma_r
$ 
be a finite Rauzy path of length $r$. Let $B_\gamma\in\slgroup(d,\ZZ)$ be the matrix defined in \S~\ref{SectionRauzyClasses}. Denote by 
$
\vec{1}\in\NN^\cA
$ 
the vector with all entries equal to $1$, define the integer vector 
$
q^{(r)}:=B_\gamma\vec{1}\in\NN^\cA
$, 
and for $\alpha\in\cA$ let $q^{(r)}_\alpha$ be its $\alpha$-entry. Denoting by $[A]_{\alpha,\beta}$ the entry of a matrix $A$ in row $\alpha$ and column $\beta$, we have
$$
q^{(r)}_\alpha=
\sum_{\chi\in\cA}[B_\gamma]_{\alpha,\chi}.
$$
Fix $f\in\cG(\pi,I)$ and $r\in\NN$ such that  
$
f^{(k)}:I^{(k)}\to I^{(k)}
$ 
satisfies Condition~\eqref{EquationConditionRenormalizableGIET} for any $k=0,\dots,r$. In particular in our notation 
$
f^{(r)}\in\cG(\pi^{(r)},I^{(r)})
$. 
Let 
$
\cP_t(f^{(r)})=\{I^t_\alpha(f^{(r)})\}_{\alpha\in\cA}
$ 
be the partition of $I^{(r)}$ associated to $f^{(r)}$ as in Equation~\eqref{EquationDynamicalPartition}, whose atoms are the sub-intervals 
$
I^t_\alpha(f^{(r)})\subset I^{(r)}
$ 
where the restriction of $f^{(r)}$ is an orientation preserving homeomorphism as in Equation~\eqref{EquationPiecewiseHomeomorphism}. In particular, if $f=T(\pi,\lambda)$ is the IET defined by combinatorial-length data $(\pi,\lambda)$ then for any $\alpha\in\cA$ we have 
$$
I^t_\alpha(f^{(r)})=
\bigg[
\sum_{\pi^{(r)}_t(\chi)\leq\pi^{(r)}_t(\alpha)-1}
\lambda^{(r)}_\chi
,
\sum_{\pi^{(r)}_t(\chi)\leq\pi^{(r)}_t(\alpha)}
\lambda^{(r)}_\chi
\bigg).
$$ 
Let $\gamma=\gamma(f,r)$ be the Rauzy path as in Definition~\ref{DefinitionRauzyPathOfGIET} and set 
$
q^{(r)}:=B_\gamma\vec{1}
$, 
which of course depends both on $f$ and on $r$. Define the intervals 

\begin{equation}
\label{EquationDynamicalPartitionHigherOrder}
I(f,r,\alpha,i):=f^i\big(I^t_\alpha(f^{(r)})\big)
\quad
\textrm{ for }
\quad
\alpha\in\cA
\quad,\quad
0\leq i\leq q^{(r)}_\alpha-1.
\end{equation}

The following Lemma is a classical fact, a proof of which can be found in \S~7.5 in \cite{YoccozClay}.

\begin{lemma}
\label{LemmaIntervalsPartitionAndMatrix}
The following holds.
\begin{enumerate}
\item
For any $\alpha\in\cA$ we have
$
f^{(r)}|_{I^t_\alpha(f,r)}=f^{q^{(r)}_\alpha}
$ 
and moreover
$$
q^{(r)}_\alpha=
\min
\left\{
i\geq1\quad;\quad f^i\big(I^t_\alpha(f^{(r)})\big)\subset I^{(r)}
\right\}.
$$
\item
For any $\alpha$ and $\beta$ we have
$$
[B_\gamma]_{\alpha,\beta}=
\sharp
\left\{
0\leq i\leq q^{(r)}_\alpha-1
\quad;\quad
I(f,r,\alpha,i)\subset I^t_\beta
\right\}.
$$
\end{enumerate}
\end{lemma}

Since $f:I\to I$ is a bijection and $f^{(r)}:I^{(r)}\to I^{(r)}$ is the first return of $f$ to 
$
I^{(r)}\subset I
$, 
then Lemma~\ref{LemmaIntervalsPartitionAndMatrix} implies that the right-open intervals $I(f,r,\alpha,i)$ defined above form a partition of $I$, that we denote by $\cP(f,r)$. Explicitly: 
$$
\cP(f,r):=
\left\{
I(f,r,\alpha,i)
\quad;\quad
\alpha\in\cA
\quad;\quad
0\leq i\leq q^{(r)}_\alpha-1
\right\}.
$$
The partition above is called the \emph{dynamical partition} of order $r$, and is a refinement of the partition $\cP_t(f)$ in Equation~\eqref{EquationDynamicalPartition}, which in our terminology corresponds to the dynamical partition of order $0$. Letting $f$ vary in $\cG(\pi,I)$ we obtain different partitions of $I$ of order $r$. For $f,f'\in\cG(\pi,I)$ we say that two dynamical partitions 
$\cP(f,r)$ and $\cP(f',r)$ are \emph{combinatorially equivalent}, and we write 
$
\cP(f,r)\sim\cP(f',r)
$, 
if the intervals $I(f,r,\alpha,i)$ in $\cP(f,r)$ are as in Equation~\eqref{EquationDynamicalPartitionHigherOrder}, moreover the intervals $I(f',r,\alpha,i)$ in $\cP(f',r)$ are labeled by the same indices $\alpha$ and $i$ as for $f$ and we have 
\begin{equation}
\label{EquationDynamicalPartitionHigherOrderEquivalence}
I(f',r,\alpha,i)=
\varphi\big(I(f,r,\alpha,i)\big)
\quad
\textrm{ for any }
\quad
\alpha\in\cA
\quad,\quad
0\leq i\leq q^{(r)}_\alpha-1,
\end{equation}
where $\varphi:I\to I$ is an increasing homeomorphism.  In other words the intervals in the two partitions have the same labels in the same order. Considering the inverse and the composition of increasing homeomorphisms of $I$, it is easy to see that the relation above is an equivalence relation.

\begin{proposition}
\label{PropositionPartitionDeterminesRauzyPath}
Let $f$ and $f'$ be two elements of $\cG(\pi,I)$ and fix $r\in\NN$. Then we have the equivalence 
$$
\cP(f,r)\sim\cP(f',r)
\quad
\Leftrightarrow
\quad
\gamma(f,r)=\gamma(f',r).
$$
\end{proposition}

\begin{proof}
We prove the Proposition by induction. For $r=0$ the statement is trivial. Consider $r=1$ and recall \S~\ref{SectionRauzyClasses} and \S~\ref{SectionRauzyInductionGIETs}. We have 
$
\gamma(f,1)=\gamma(f',1)
$ 
iff 
$
\epsilon(f)=\epsilon(f')
$. 
If $\epsilon=t$ then $q^{(1)}_\chi=2$ for $\chi=\alpha_b$ and $q^{(1)}_\chi=1$ for any 
$\chi\not=\alpha_b$, moreover 
\begin{eqnarray*}
&&
I(f,1,\chi,0)=I^t_\chi(f^{(1)})=I^t_\chi(f)
\quad
\textrm{ for }
\quad
\chi\not=\alpha_t
\\
&&
I(f,1,\alpha_t,0)=I^t_{\alpha_t}(f^{(1)})
=
I^t_{\alpha_t}(f)\setminus f\big(I^t_{\alpha_b}(f)\big)
\\
&&
I(f,1,\alpha_b,1)=f\big(I^t_{\alpha_b}(f)\big),
\end{eqnarray*}
and the same holds for $f'$, therefore $\cP(f,1)\sim\cP(f',1)$. On the other hand, if $\epsilon=b$ then $q^{(1)}_\chi=2$ for $\chi=\alpha_t$ and $q^{(1)}_\chi=1$ for any 
$\chi\not=\alpha_t$, moreover 
\begin{eqnarray*}
&&
I(f,1,\chi,0)=I^t_\chi(f^{(1)})=I^t_\chi(f)
\quad
\textrm{ for }
\quad
\chi\not=\alpha_t,\alpha_b
\\
&&
I(f,1,\alpha_b,0)=I^t_{\alpha_b}(f^{(1)})=
I^t_{\alpha_b}(f)\setminus f^{-1}\big(I^t_{\alpha_t}(f)\big)
\\
&&
I(f,1,\alpha_t,0)=I^t_{\alpha_t}(f^{(1)})=f^{-1}\big(I^t_{\alpha_t}(f)\big)
\\
&&
I(f,1,\alpha_t,1)=f\big(I^t_{\alpha_t}(f^{(1)})\big)=I^t_{\alpha_t}(f),
\end{eqnarray*}
and the same holds for $f'$, therefore again $\cP(f,1)\sim\cP(f',1)$. The argument above proves the implication 
$
\gamma(f,1)=\gamma(f',1)\Rightarrow\cP(f,1)\sim\cP(f',1)
$, 
but since there are only two equivalence classes of dynamical partition of order $r=1$, then the opposite implication also holds. Finally fix $r\in\NN$ as assume that the equivalence in the statement is proved up to $r$ for any $\pi'\in\cR$, any interval $J$ and any $f,f'\in\cG(\pi',J)$. Fix 
$f,f'\in\cG(\pi,I)$. Observe preliminarly that for a map $f:I\to I$ the intervals in the partition 
$\cP(f,r+1)$ of $I$ are the images under $f$ of the intervals in the partition $\cP(f^{(r)},1)$ of $I^{(r)}$. Then the Proposition follows because the inductive assumption, together with the argument for $r=1$, imply the equivalences
\begin{align*}
\gamma(f,r+1)=\gamma(f',r+1)
&
\Leftrightarrow
\left\{
\begin{array}{l}
\gamma(f,r)=\gamma(f',r)
\\
\epsilon(f^{(r)})=\epsilon\big((f')^{(r)}\big)
\end{array}
\right.
\\
&
\Leftrightarrow
\left\{
\begin{array}{l}
\cP(f,r)\sim\cP(f',r)
\\
\cP(f^{(r)},1)\sim\cP((f')^{(r)},1)
\end{array}
\right.
&
\Leftrightarrow
\cP(f,r+1)\sim\cP(f',r+1).
\end{align*}
\end{proof}

\section{Proof of the Full Family Theorem}
\label{SectionProofFullFamilyTheorem}

In this section we prove the Full Family Theorem~\ref{TheoremFullFamily}, assuming Theorem~\ref{TheoremRealizationRauzyPath} below, which is the main technical result in this paper. The argument is given in \S~\ref{SectionProofTheoremFullFamily} below. Besides this section \S~\ref{SectionProofFullFamilyTheorem}, the rest of the paper is devoted to the development of the tools needed to prove Theorem~\ref{TheoremRealizationRauzyPath}, whose proof is resumed in \S~\ref{SectionProofTheoremRealizationRauzyPath}. Fix an admissible combinatorial datum $\pi$, let $\cR$ be its Rauzy class, and assume that $\cR$ contains also a cyclic combinatorial datum $\pi^{(\ast)}$, introduced in Definition~\ref{DefinitionCyclicCombinatorialDatum}. Recall from Definition~\ref{DefinitionFullFamily} the notion of full family 
$
(f_\tau)_{\tau\in\Delta^\cA}\subset \cG\big(\pi,[0,1)\big)
$ 
over the admissible combinatorial datum $\pi$. 

\begin{theorem}
\label{TheoremRealizationRauzyPath}
Let $\pi$ be an admissible combinatorial datum and $\cR$ be its Rauzy class, and assume that $\cR$ contains a cyclic combinatorial datum $\pi^{(\ast)}$. Let  
$
(f_\tau)_{\tau\in\Delta^\cA}
$ 
be a full family with combinatorial datum $\pi$. Then for any finite Rauzy path $\gamma:\pi\to\pi^{(\ast)}$ of length $r$ there exists 
$
\tau=\tau(\gamma)\in\Delta^\cA
$ 
such that 
$$
\gamma(f_{\tau(\gamma)},r)=\gamma.
$$
\end{theorem}

Given any two finite paths $\gamma:\pi\to\pi'$ and $\eta:\pi'\to\pi''$ in $\cR$, where the starting point $\pi'\in\cR$ of $\eta$ coincides with the ending point of $\gamma$, consider the concatenation 
$
\gamma\ast\eta:\pi\to\pi''
$.

\begin{remark}
\label{RemarkRealizationRauzyPath}
The requirement that $\gamma:\pi\to\pi^{(\ast)}$ ends at $\pi^{(\ast)}$ in Theorem~\ref{TheoremRealizationRauzyPath} can be easily removed a posteriori. Indeed if $\gamma:\pi\to\pi'$ is any finite Rauzy path with length $r$, there exists a finite path 
$
\eta:\pi'\to\pi^{(\ast)}
$ 
starting at the ending point $\pi'$ of $\gamma$, ending at $\pi^{(\ast)}$ and with length $l(\eta)\leq|\cR|$, where $|\cR|$ denotes the cardinality of $\cR$. For the concatenation $\gamma\ast\eta$ Theorem~\ref{TheoremRealizationRauzyPath} gives $\tau\in\Delta^\cA$ with 
$$
\gamma\big(f_\tau,r+l(\eta)\big)=\gamma\ast\eta,
$$ 
so that by truncation we get $\gamma(f_\tau,r)=\gamma$.
\end{remark}

\subsection{Preliminary Lemmas}

Before passing to the proof of Theorem~\ref{TheoremFullFamily}, we need Lemma~\ref{LemmaContinuityRauzyPath} and Lemma~\ref{LemmaPositiveRauzyPaths} below.

\begin{lemma}
\label{LemmaContinuityRauzyPath}
Fix any $r\in\NN$ and consider $f\in\cG\big(\pi,[0,1)\big)$ which admits $r$ steps of the Rauzy induction. Then there exists a set  
$
\cU=\cU(f,r)\subset\cG\big(\pi,[0,1)\big)
$ 
open with respect to the distance defined by Equation~\eqref{EquationDinstanceGIETs} such that for any 
$f'\in\cU$ we have 
$$
\gamma(f,r)=\gamma(f',r),
$$
that is $f$ and $f'$ have the same Rauzy renormalization path up to the first $r$ steps.
\end{lemma}

\begin{proof}
Recall the notation in \S~\ref{SectionDynamicalPartition}. Set $\gamma=\gamma(f,r)$ and let 
$
q^{(r)}=B_\gamma\vec{1}
$ 
be the corresponding vector of return times. Let $\cP(f,r)$ be the corresponding dynamical partition, and recall that its atoms are intervals whose endpoints have the form $f^k(u^t_\alpha)$ with $\alpha\in\cA$ and 
$
|k|\leq \max_{\chi\in\cA}q_\chi^{(r)}
$. 
All these points depend continuously on $f\in\cG\big(\pi,[0,1)\big)$, hence we have 
$
\cP(f',r)\sim\cP(f,r)
$ 
if $f'$ is close enough to $f$. Then the Lemma follows directly from Proposition~\ref{PropositionPartitionDeterminesRauzyPath}. 
\end{proof}

Fix a full family 
$
(f_\tau)_{\tau\in\Delta^\cA}
$ 
over $\pi$, parametrized by a map 
$
\Delta^\cA\to\cG\big(\pi,[0,1)\big)
$, 
$
\tau\mapsto f_\tau
$. 
For any finite Rauzy path $\gamma:\pi\to\pi'$ of length $r$ denote by $\Delta_\gamma$ the set of those $\tau\in\Delta^\cA$ such that $f_\tau$ admits $r$ steps of the Rauzy induction map $\cQ$ with  
$
\gamma(f_\tau,r)=\gamma
$. 
The set $\Delta_\gamma$ is open in $\Delta^\cA$, according to Lemma~\ref{LemmaContinuityRauzyPath} and to the continuity of the map $\tau\mapsto f_\tau$. Moreover such open set $\Delta_\gamma$ is not empty, according to Theorem~\ref{TheoremRealizationRauzyPath} and Remark~\ref{RemarkRealizationRauzyPath}, and we denote by 
$
\overline{\Delta_\gamma}
$ 
its closure in $\overline{\Delta^\cA}$. A finite Rauzy path $\eta$ in $\cR$ is \emph{positive} if the matrix $B_\eta\in\slgroup(d,\ZZ)$ defined in \S~\ref{SectionRauzyClasses} has all its entries strictly positive. This  notion has a topological counterpart for full families of GIETs established by the Lemma below.

\begin{lemma}
\label{LemmaPositiveRauzyPaths}
Consider finite paths $\gamma:\pi\to\pi'$ and $\eta:\pi'\to\pi''$ as above and assume that $\eta$ is positive. Then we have
$$
\overline{\Delta_{\gamma\ast\eta}}\subset\Delta_\gamma.
$$
\end{lemma}

\begin{proof}
Recall the notation in \S~\ref{SectionDynamicalPartition}. Let $r$ and $k$ be the length of $\gamma$ and $\eta$ respectively, so that $\gamma\ast\eta$ has length $r+k$, then consider the vectors 
$
q^{(r)}:=B_{\gamma}\vec{1}
$, 
$
q^{(r+k)}:=B_{\gamma\ast\eta}\vec{1}
$ 
and 
$
Q^{(k)}:=B_{\eta}\vec{1}
$ 
in $\NN^\cA$. Let $T=T(\pi,\lambda)$ be a fixed IET such that $\gamma(T,r+k)=\gamma\ast\eta$. Consider the partition $\cP(T,r)$ and its refinement $\cP(T,r+k)$. According to Propositin~\ref{PropositionPartitionDeterminesRauzyPath}, for any 
$
\tau\in\Delta_{\gamma\ast\eta}
$ 
similar partitions $\cP(f_\tau,r)$ and $\cP(f_\tau,r+k)$ are defined, where 
$
\cP(f_\tau,r)\sim\cP(T,r)
$ 
and 
$
\cP(f_\tau,r+k)\sim\cP(T,r+k)
$, 
and where the atoms in $\cP(f_\tau,r)$ are the intervals 
$$
I(f_{\tau},r,\chi,i):=
f_{\tau}^i\big(I^t_\chi(f_{\tau}^{(r)})\big)
\quad
\textrm{ for }
\quad
\chi\in\cA
\quad;\quad
0\leq i\leq q^{(r)}_\chi-1
$$
and the atoms in $\cP(f_\tau,r+k)$ are the intervals  
$$
I(f_{\tau},r+k,\chi,i):=
f_{\tau}^i\big(I^t_\chi(f_{\tau}^{(r+k)})\big)
\quad
\textrm{ for }
\quad
\chi\in\cA
\quad;\quad
0\leq i\leq q^{(r+k)}_\chi-1.
$$
Consider 
$
\tau(\infty)\in\partial\Delta_{\gamma\ast\eta}
$ 
and let 
$
\tau(n)\in\Delta_{\gamma\ast\eta}
$ 
with $\tau(n)\to\tau(\infty)$. A priori $f_{\tau(\infty)}$ is a degenerate GIET, but we show that indeed 
$
\tau(\infty)\in\Delta_{\gamma}
$, 
which in particular implies that 
$
f_{\tau(\infty)}\in\cG\big(\pi,[0,1)\big)
$. 
Since any $\cP(f_{\tau(n)},r+k)$ is a partition of $[0,1)$ and we have 
$
f_{\tau(n)}\to f_{\tau(\infty)}
$ 
as $n\to\infty$, then there exists a letter $\alpha\in\cA$ such that 
$$
\inf_{n\in\NN}|I(f_{\tau(n)},r+k,\alpha,i)|>0
\quad
\textrm{ for any}
\quad
0\leq i\leq q^{(r+k)}_\alpha-1.
$$
Since 
$
\gamma(f^{(r)}_{\tau(n)},k)=\eta
$ 
is positive, where 
$
f_{\tau(n)}^{(r)}=\cQ^r(f_{\tau(n)})
$ 
is the $r$-th step of the Rauzy map applied to $f_{\tau(n)}$, then Lemma~\ref{LemmaIntervalsPartitionAndMatrix} implies that there exists an integer 
$j$ with 
$0\leq j\leq Q_\alpha^{(k)}-1$ such that  
$$
(f_{\tau(n)}^{(r)})^j\big(I^t_\alpha(f_{\tau(n)}^{(r+k)})\big)
\subset I^t_\beta(f_{\tau(n)}^{(r)})
\quad
\textrm{ for any }
\quad
n\in\NN.
$$
Recalling that $f_{\tau(n)}^{(r)}$ is a first return of $f_{\tau(n)}$, then for any $\beta\in\cA$ and any $i$ with 
$
0\leq i\leq q^{(r)}_\beta-1
$ 
there exists $l$ with 
$
0\leq l\leq q_\alpha^{(r+k)}-1
$ 
such that
$$
I(f_{\tau(n)},r+k,\alpha,l)=
f_{\tau(n)}^l\big(I^t_\alpha(f_{\tau(n)}^{(r+k)})\big)
\subset I(f_{\tau(n)},r,\beta,i).
$$
It follows that
$$
\inf_{n\in\NN}|I(f_{\tau(n)},r,\beta,i)|>0
\quad
\textrm{ for any}
\quad
\beta\in\cA
\quad
\textrm{ and any }
\quad
0\leq i\leq q^{(r)}_\beta-1.
$$
Since 
$
\cP(f_{\tau(n)},r)\sim\cP(T,r)
$,  
then for any $n$ the intervals $I(f_{\tau(n)},r,\beta,i)$ in $\cP(f_{\tau(n)},r)$ have the same order as the intervals $I(T,r,\beta,i)$ in $\cP(T,r)$. Moreover the last condition implies that their size stays bounded from below for any $n$. If follows that the map $f_{\tau(\infty)}$ admits a partition 
$
\cP(f_{\tau(\infty)},r)
$ 
whose atoms are the non-degenerate right-open intervals defined by
$$
I(f_{\tau(\infty)},r,\beta,i):=
\lim_{n\to\infty}I(f_{\tau(n)},r,\beta,i),
$$
where the limit above is in the Hausdorff metric. Since the order is preserved in the limit, for such partition we have 
$
\cP(f_{\tau(\infty)},r)\sim\cP(T,r)
$, 
which is equivalent to 
$
\gamma(f_{\tau(\infty)},r)=\gamma(T,r)
$ 
according to Proposition~\ref{PropositionPartitionDeterminesRauzyPath}. The Lemma is proved.
\end{proof}

\subsection{Proof of Theorem~\ref{TheoremFullFamily}}
\label{SectionProofTheoremFullFamily}

Let $\pi$ be a combinatorial datum as in Theorem~\ref{TheoremFullFamily}. Let $\lambda\in\Delta^\cA$ be a length datum and $T=T(\pi,\lambda)$ be the corresponding IET, which we assume to be infinitely renormalizable, according to the definitions introduced in \S~\ref{SectionIterationRauzyInductionGIETs}. Recalling Definition~\ref{DefinitionRauzyPathOfGIET}, 
let 
$
\gamma(T,\infty)
$ 
be the infinite Rauzy path of $T$. Then fix $r\in\NN$ and let 
$
\gamma(T,r)=\gamma_1\ast\dots\ast\gamma_r
$ 
be the concatenation of the first $r$ arrows of 
$\gamma(T,\infty)$. According to Theorem~\ref{TheoremRealizationRauzyPath} and Remark~\ref{RemarkRealizationRauzyPath} there exists  
$
\tau(r)\in\Delta^\cA
$ 
such that 
\begin{equation}
\label{Equation(1)SectionProofTheoremFullFamily}
\gamma(f_{\tau(r)},r)=\gamma(T,r). 
\end{equation}
Moreover, according to Lemma~\ref{LemmaContinuityRauzyPath} and to the continuity of the map $\tau\mapsto f_\tau$, there is a non-empty open set 
$$
\Delta_{\gamma(T,r)}:=\{\tau\in\Delta^\cA,\gamma(f_\tau,r)=\gamma(T,r)\}.
$$
Increasing $r$ we get a sequence of nested non-empty open sets 
$$
\dots\subset\Delta_{\gamma(T,r+1)}\subset\Delta_{\gamma(T,r)}\subset\dots\subset\Delta^\cA.
$$
According to Proposition 7.9 in \cite{YoccozClay}, the path $\gamma(T,\infty)$ is infinite-complete, that is any letter $\alpha\in\cA$ is the winner of infinitely many arrows of $\gamma(T,\infty)$ (see also \S~1.2.3 in \cite{MarmiMoussaYoccozCohomological}). As a consequence of this last property, according to Proposition 7.12 in \cite{YoccozClay}, for any $r\in\NN$ there is an integer $k=k(r)>r$ such that, decomposing $\gamma(T,k)$ as
$$
\gamma(T,k)=\gamma(T,r)\ast\eta(T,r,k),
$$
the factor $\eta(T,r,k)$ is a positive finite Rauzy path (see also \S~1.2.4 in \cite{MarmiMoussaYoccozCohomological}). Therefore Lemma~\ref{LemmaPositiveRauzyPaths} implies that 
$$
\Delta_{\gamma(T,\infty)}:=
\bigcap_{r\in\NN}\overline{\Delta_{\gamma(T,r)}}
$$
is a compact subset contained in the interior of $\Delta^\cA$. Therefore, modulo taking a subsequence of the parameters $\tau(r)\in\Delta^\cA$, $r\in\NN$ in Equation~\eqref{Equation(1)SectionProofTheoremFullFamily}, there exists $\tau\in\Delta_{\gamma(T,\infty)}$ such that $\tau(r)\to\tau$ as $r\to\infty$, that is 
$
f_{\tau(r)}\to f_\tau
$. 
In particular the limit $f_\tau$ is a non-degenerate element in 
$
\cG\big(\pi,[0,1)\big)
$, 
because $\tau$ belongs to the interior of $\Delta^\cA$. Moreover $f_\tau$ is infinitely renormalizable, because 
$
\tau\in\Delta_{\gamma(T,\infty)}
$. 
Fix $r_0\in\NN$ and let 
$
\cU=\cU(f_\tau,r_0)\subset\cG\big(\pi,[0,1)\big)
$ 
be the open set around $f_\tau$ as in Lemma~\ref{LemmaContinuityRauzyPath}. Since 
$
f_{\tau(r)}\to f_\tau
$ 
as $r\to\infty$, then for $r$ big enough we have $f_{\tau(r)}\in \cU$. Moreover any $r$ big enough satisfies also $r\geq r_0$. Thus, according to Lemma~\ref{LemmaContinuityRauzyPath}, for any $r$ big enough we have
$$
\gamma(f_\tau,r_0)=
\gamma(f_{\tau(r)},r_0)=
\gamma(T,r_0).
$$
Since $r_0$ in the last equality can be chosen arbitrarily big, we get 
$
\gamma(f_\tau,\infty)=\gamma(T,\infty)
$. 
Theorem~\ref{TheoremFullFamily} is proved $\qed$.

\section{The Thurston map}
\label{SectionThurstonMap}

Fix an admissible combinatorial datum $\pi$ over the alphabet $\cA$, let $\cR$ be its Rauzy class and assume that $\cR$ contains a cyclic combinatorial datum $\pi^{(\ast)}$, where we recall that this means that the permutation 
$
\sigma(\pi^{(\ast)}):\{1,\dots,d\}\to\{1,\dots,d\}
$ 
introduced in Definition~\ref{DefinitionCyclicCombinatorialDatum} is cyclic with maximal order $d$. Fix a finite Rauzy path $\gamma:\pi\to\pi^{(\ast)}$ and let $r$ be its length. Let $B_\gamma\in\slgroup(d,\ZZ)$ be the matrix defined in \S~\ref{SectionRauzyClasses}. Recall that $\vec{1}\in\RR^\cA$ denotes the vector with all entries equal to $1$ and let 
$
q^{(r)}:=B_\gamma\vec{1}
$ 
be the integer vector of \emph{return times} defined in \S~\ref{SectionDynamicalPartition}. Define the positive integer
$$
N=N(\gamma):=\sum_{\chi\in\cA}q^{(r)}_\chi.
$$

\subsection{Configurations}
\label{SectionConfigurations}

Consider the normalized length datum 
$
\lambda^{(\gamma)}\in\QQ_+^\cA\cap\Delta^\cA
$ 
defined by 
$$
\lambda^{(\gamma)}:=N^{-1}\cdot ^t\!B_\gamma\vec{1}.
$$ 
Observe that 
$
\sum_{\chi\in\cA}\lambda^{(\gamma)}_\chi=1
$ 
and let 
$
T_\gamma:[0,1)\to[0,1)
$ 
be the IET with length datum $\pi$ and combinatorial datum $\lambda^{(\gamma)}$, that is 
$
T_\gamma:=T(\pi,\lambda_\gamma)
$. 
Let
$
T_\gamma^{(r)}:=\cQ^r(T_\gamma)
$ 
be the image of $T_\gamma$ under the $r$-th iteration of the Rauzy map, which is defined in \S~\ref{SectionRauzyInductionGIETs}. Recalling \S~\ref{SectionDynamicalPartition}, for $\alpha\in\cA$ consider the intervals $I^t_\alpha(T_\gamma^{(r)})$ and their images under $T_\gamma$, that is the intervals $I(T_\gamma,r,\alpha,i)$ with $0\leq i\leq q^{(r)}_\alpha-1$ as in Equation~\eqref{EquationDynamicalPartitionHigherOrder}, which are the atoms of the partition $\cP(T_\gamma,r)$. Observe that
$$
T_\gamma^{(r)}=T(\pi^{(\ast)},N^{-1}\cdot \vec{1}),
$$
which is a periodic IET acting on the intervals $I^t_\alpha(T_\gamma^{(r)})$, $\alpha\in\cA$ as the cyclic permutation 
$
\sigma(\pi^{(\ast)})
$. 
Since $T_\gamma^{(r)}$ is a first return of $T_\gamma$, then $T_\gamma$ acts as a cyclic permutation on the intervals of the partition $\cP(T_\gamma,r)$ by 
$$
T_\gamma\big(I(T_\gamma,r,\alpha,i)\big)=
I(T_\gamma,r,\alpha,i+1),
$$
modulo identifications on the labels $(\alpha,i)\in\cA\times\ZZ$ given by  
$$
(\alpha,i)\sim(\beta,j)
\Leftrightarrow
\left\{
\begin{array}{l}
\pi^{(\ast)}_t(\beta)=\pi^{(\ast)}_b(\alpha)
\\
i=q^{(r)}_\alpha +j.
\end{array}
\right.
$$
It is convenient to consider different identifications on $\cA\times\ZZ$. Recall from \S~\ref{SectionStatementOfResults} that we call  
$
u^t_\alpha(T_\gamma):=\min I^t_\alpha(T_\gamma)
$ 
and 
$
u^t_\alpha(T_\gamma^{(r)}):=\min I^t_\alpha(T_\gamma^{(r)})
$ 
the critical points respectively of $T_\gamma$ and of $T^{(r)}$ corresponding to the letter $\alpha\in\cA$, then let $h(\alpha,r)$ be the unique integer with 
$
0\leq h(\alpha,r)\leq q^{(r)}_\alpha-1
$ 
such that 
$$
T_\gamma^{h(\alpha,r)}\big(u^t_\alpha(T_\gamma^{(r)})\big)=
u^t_\alpha(T_\gamma).
$$
In particular $h(\alpha_0,r)=0$ for the letter $\alpha_0$ with $\pi_t(\alpha_0)=1$. Consider the identifications between the labels in $\cA\times\ZZ$ given by
\begin{equation}
\label{EquationIdentificationLabelsCriticalPoints}
(\alpha,i)\sim(\beta,j)
\Leftrightarrow
\left\{
\begin{array}{l}
\pi_t(\beta)=\pi_b(\alpha)
\\
i-j=q^{(r)}_\alpha +h(\beta,r)-h(\alpha,r).
\end{array}
\right.
\end{equation}
Let $\cI(\gamma)$ be the quotient of the set $\cA\times\ZZ$ under the equivalence relation $\sim$ induced by Equation~\eqref{EquationIdentificationLabelsCriticalPoints}, whose equivalence classes are denoted by $[\alpha,i]$. The map 
$
[\alpha,i]\mapsto[\alpha,i+1]
$ 
on equivalence classes still corresponds to the cyclic action of $T_\gamma$ on intervals of the partition $\cP(T_\gamma,r)$, but in this notation, for $\alpha\in\cA$, the classes $[\alpha,0]$ correspond to the intervals 
$
I\big(T_\gamma,r,\alpha,h(\alpha,r)\big)
$. 
An example of such labelling is given in Figure \S~\ref{FigureExampleConfiguration}, where we consider the alphabet $\cA:=\{A,B,C,D\}$ and the admissible combinatorial data 
$$
\pi=
\begin{pmatrix}
A & B & C & D \\
D & C & B & A 
\end{pmatrix}
\quad
\textrm{ and }
\quad
\pi^{(\ast)}=
\begin{pmatrix}
A & B & D & C \\
D & A & C & B 
\end{pmatrix},
$$
which belong to the same Rauzy class, and are connected by the Rauzy path $\gamma:\pi\to\pi^{(\ast)}$ of length five whose winner are in the order the letters $A,A,A,D,B$. In particular $\pi^{(\ast)}$ is cyclic with 
$
\sigma(\pi^{(\ast)})=(1,2,4,3)
$, 
which has maximal order $4$ in the symmetric group $S_4$. In this case, the matrix $B_\gamma$ defined in \S~\ref{SectionRauzyClasses} is 
$$
B_{\gamma}=
\begin{pmatrix}
2 & 0 & 0 & 1 \\
1 & 1 & 0 & 0 \\
1 & 0 & 1 & 0 \\
2 & 1 & 0 & 1
\end{pmatrix}
\quad
\textrm{ with }
\quad
^tB^{-1}_{\gamma}
=
\begin{pmatrix}
1 &-1 &-1 &-1 \\
1 & 0 &-1 &-2 \\
0 & 0 & 1 & 0 \\
-1& 1 & 1 & 2
\end{pmatrix}.
$$
We have $q^{(r)}=(3,2,2,4)$ and $N(\gamma)=11$, so that 
$
\lambda^{(\gamma)}=(\frac{6}{11},\frac{2}{11},\frac{1}{11},\frac{2}{11})
$. 
Moreover the vector $h\in\NN^\cA$ whose $\alpha$-entry is $h_\alpha:=h(\alpha,r)$ is $h=(0,1,1,3)$.

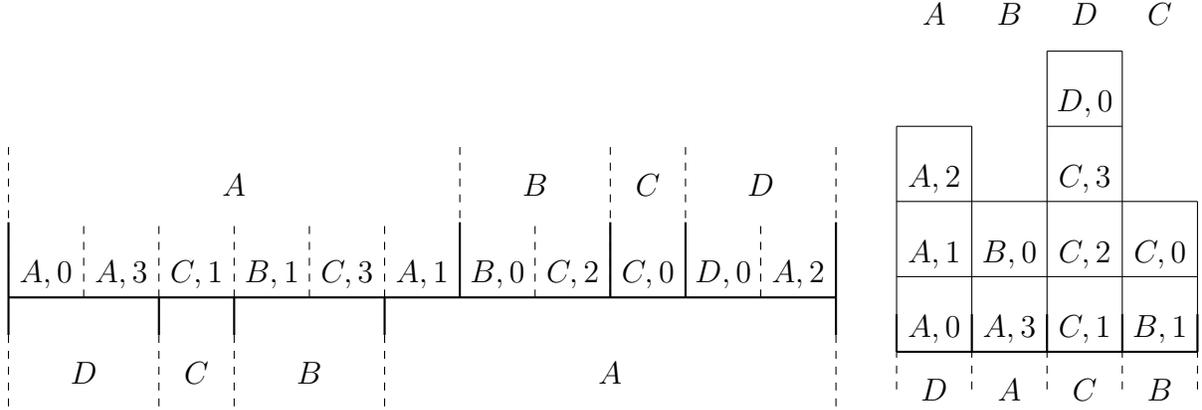
\begin{figure}
\begin{center}
{\begin{tikzpicture}[scale=0.1]



\draw[-,thick] (0,0) -- (110,0) {};
\draw[-,thick] (0,0) -- (0,10) {};
\draw[-,thin,dashed] (0,0) -- (0,20) {};
\node at (30,15) {$A$};
\node at (5,3) {$A,0$};
\draw[-,thin,dashed] (10,0) -- (10,10) {};
\node at (15,3) {$A,3$};
\draw[-,thin,dashed] (20,0) -- (20,10) {};
\node at (25,3) {$C,1$};
\draw[-,thin,dashed] (30,0) -- (30,10) {};
\node at (35,3) {$B,1$};
\draw[-,thin,dashed] (40,0) -- (40,10) {};
\node at (45,3) {$C,3$};
\draw[-,thin,dashed] (50,0) -- (50,10) {};
\node at (55,3) {$A,1$};

\draw[-,thick] (60,0) -- (60,10) {};
\draw[-,thin,dashed] (60,0) -- (60,20) {};
\node at (70,15) {$B$};
\node at (65,3) {$B,0$};
\draw[-,thin,dashed] (70,0) -- (70,10) {};
\node at (75,3) {$C,2$};

\draw[-,thick] (80,0) -- (80,10) {};
\draw[-,thin,dashed] (80,0) -- (80,20) {};
\node at (85,15) {$C$};
\node at (85,3) {$C,0$};

\draw[-,thick] (90,0) -- (90,10) {};
\draw[-,thin,dashed] (90,0) -- (90,20) {};
\node at (100,15) {$D$};
\node at (95,3) {$D,0$};
\draw[-,thin,dashed] (100,0) -- (100,10) {};
\node at (105,3) {$A,2$};
\draw[-,thick] (110,0) -- (110,10) {};
\draw[-,thin,dashed] (110,0) -- (110,20) {};


\draw[-,thick] (0,0) -- (0,-5) {};
\draw[-,thin,dashed] (0,-5) -- (0,-15) {};
\node at (10,-10) {$D$};
\draw[-,thick] (20,0) -- (20,-5) {};
\draw[-,thin,dashed] (20,-5) -- (20,-15) {};
\node at (25,-10) {$C$};
\draw[-,thick] (30,0) -- (30,-5) {};
\draw[-,thin,dashed] (30,-5) -- (30,-15) {};
\node at (40,-10) {$B$};
\draw[-,thick] (50,0) -- (50,-5) {};
\draw[-,thin,dashed] (50,-5) -- (50,-15) {};
\node at (80,-10) {$A$};
\draw[-,thick] (110,0) -- (110,-5) {};
\draw[-,thin,dashed] (110,-5) -- (110,-15) {};

\end{tikzpicture}}
\hspace{0.5 cm}
{\begin{tikzpicture}[scale=0.1]



\draw[-,thick] (0,0) -- (40,0) {};
\draw[-,thick] (0,0) -- (0,5) {};
\node at (5,45) {$A$};
\draw[-,thick] (10,0) -- (10,5) {};
\node at (15,45) {$B$};
\draw[-,thick] (20,0) -- (20,5) {};
\node at (25,45) {$D$};
\draw[-,thick] (30,0) -- (30,5) {};
\node at (35,45) {$C$};
\draw[-,thick] (40,0) -- (40,5) {};

\draw[-,thin,dashed] (0,-5) -- (0,0) {};
\node at (5,-5) {$D$};
\draw[-,thin,dashed] (10,-5) -- (10,0) {};
\node at (15,-5) {$A$};
\draw[-,thin,dashed] (20,-5) -- (20,0) {};
\node at (25,-5) {$C$};
\draw[-,thin,dashed] (30,-5) -- (30,0) {};
\node at (35,-5) {$B$};
\draw[-,thin,dashed] (40,-5) -- (40,0) {};

\draw[-,thin] (0,0) -- (0,30) {};
\node at (5,3) {$A,0$};
\draw[-,thin] (0,10) -- (10,10) {};
\node at (5,13) {$A,1$};
\draw[-,thin] (0,20) -- (10,20) {};
\node at (5,23) {$A,2$};
\draw[-,thin] (0,30) -- (10,30) {};
\draw[-,thin] (10,0) -- (10,30) {};

\draw[-,thin] (10,10) -- (20,10) {};
\node at (15,3) {$A,3$};
\draw[-,thin] (10,20) -- (20,20) {};
\node at (15,13) {$B,0$};

\draw[-,thin] (20,0) -- (20,40) {};
\node at (25,3) {$C,1$};
\draw[-,thin] (20,10) -- (30,10) {};
\node at (25,13) {$C,2$};
\draw[-,thin] (20,20) -- (30,20) {};
\node at (25,23) {$C,3$};
\draw[-,thin] (20,30) -- (30,30) {};
\node at (25,33) {$D,0$};
\draw[-,thin] (20,40) -- (30,40) {};
\draw[-,thin] (30,0) -- (30,40) {};

\draw[-,thin] (30,10) -- (40,10) {};
\node at (35,3) {$B,1$};
\draw[-,thin] (30,20) -- (40,20) {};
\node at (35,13) {$C,0$};
\draw[-,thin] (40,0) -- (40,20) {};

\end{tikzpicture}}
\end{center}
\caption{The dynamical partition $\cP(T_\gamma,5)$, where 
$
T_\gamma=T(\pi,\lambda_\gamma)
$ 
in terms of $\pi$, $\lambda^{(\gamma)}$ and $\gamma:\pi\to\pi^{(\ast)}$ defined above in \S~\ref{SectionConfigurations}. The left endpoints of the atoms of $\cP(T_\gamma,5)$ are the points in the standard $\gamma$-configuration $\cV^{(\gamma)}$. They form an unique orbit of length $N=11$ under the action of $T_\gamma$. The points $\underline{v([\alpha,0])}$ for $\alpha\in\cA$ are the critical points of $T_\gamma$.}
\label{FigureExampleConfiguration}
\end{figure}

\medskip

For any $[\alpha,i]\in\cI(\gamma)$ define the rational point 
$$
\underline{v([\alpha,i])}:=
T_\gamma^i\big(u^t_\alpha(T_\gamma)\big)\in [0,1)\cap\QQ,
$$
so that in particular the points $\underline{v([\alpha,0])}$ and $\underline{v([\alpha,1])}$ for $\alpha\in\cA$ correspond respectively to the critical values and the critical points of $T_\gamma$. The \emph{reference $\gamma$-configuration} $\cV^{(\gamma)}$ is the set of $N$ points in $[0,1)\cap\QQ$ defined by 
$$
\cV^{(\gamma)}:=
\{\underline{v([\alpha,i])};[\alpha,i]\in\cI(\gamma)\}.
$$
In particular, modulo the equivalence relation on $\cA\times\ZZ$ established by Equation~\eqref{EquationIdentificationLabelsCriticalPoints} we have 
$$
T_\gamma\big(\underline{v([\alpha,i])}\big)=
\underline{v([\alpha,i+1])}.
$$

\begin{definition}
\label{DefinitionGammaConfiguration}
A $\gamma$-configuration is a set $\cV\subset[0,1)$ of $N$ distinct points $v([\alpha,i])$ labeled by classes  
$
[\alpha,i]\in\cI(\gamma)
$, 
where $v([\alpha_0,0])=0$ for the letter $\alpha_0$ with $\pi_t(\alpha_0)=1$ and moreover 
$$
v([\alpha,i])<v([\beta,j])
\quad
\Leftrightarrow
\quad
\underline{v([\alpha,i])}<\underline{v([\beta,j])}
$$
for any other $[\alpha,i]$ and $[\beta,j]$ in $\cI(\gamma)$. The \emph{configuration space} 
$
\cO(\gamma)
$ 
is the set of all $\gamma$-configurations $\cV$.
\end{definition}

In other words, the points in a $\gamma$-configuration $\cV$ have the same geometrical order as the points in the standard $\gamma$-configuration $\cV^{(\gamma)}$. In particular the latter is a $\gamma$-configuration. Recalling that $\Delta^N$ denotes the open standard simplex in $\RR^N$, and denoting by 
$
\{\cV^{(\gamma)}\}
$ 
the singleton whose only element is the standard $\gamma$-configuration $\cV^{(\gamma)}$, we have 
$$
\cO(\gamma)=\{\cV^{(\gamma)}\}\times\Delta^N.
$$

\subsection{Definition of the Thurston map} 
\label{SectionDefinitionThurstonMap}

We define a map  
$
\cO(\gamma)\to\Delta^\cA
$, 
$
\cV\mapsto\tau=\tau(\cV)
$, 
where for any $\cV\in\cO(\gamma)$ and $\alpha\in\cA$ the entry $\tau_\alpha$ of the vector $\tau=\tau(\cV)\in\Delta^\cA$ is given by 
\begin{equation}
\label{EquationTauFunctionOfConfiguration}
\left\{
\begin{array}{l}
\tau_\alpha:=1-v([\alpha,1])
\quad
\textrm{ if }
\quad
\pi_t(\alpha)=d
\\
\tau_\alpha:=v([\alpha^\ast,1])-v([\alpha,1])
\quad
\textrm{ where }
\quad
\pi_t(\alpha^\ast)=\pi_t(\alpha)+1.
\end{array}
\right.
\end{equation}

Fix a full family 
$
(f_\tau)_{\tau\in\Delta^\cA}
$ 
of GIETs over $\pi$, parametrized by a map 
$
F:\Delta^\cA\to\cG\big(\pi,[0,1)\big)
$ 
as in Definition~\ref{DefinitionFullFamily}. Fix a configuration $\cV\in\cO(\gamma)$ and let $\tau(\cV)\in\Delta^\cA$ be given by Equation~\eqref{EquationTauFunctionOfConfiguration}. Property (2) in Definition~\ref{DefinitionFullFamily}, together with Equation~\eqref{EquationMarkingCriticalValues} and Equation~\eqref{EquationTauFunctionOfConfiguration}, imply that 
$
f_{\tau(\cV)}:[0,1)\to[0,1)
$ 
is the only element in 
$
(f_\tau)_{\tau\in\Delta^\cA}
$ 
having its critical values at points $v([\alpha,1])$ with $\alpha\in\cA$. Let
$
\cV'=\{v'([\alpha,i]);[\alpha,i]\in\cI(\gamma)\}\subset[0,1)
$ 
be the set of points defined by 
\begin{equation}
\label{EquationDefinitionThurstonMap}
v'([\alpha,i]):=f_\tau^{-1}\big(v([\alpha,i+1])\big).
\end{equation}

\begin{remark}
\label{RemarkIdentificationlabels}
Recall that labels $[\alpha,i]$ represent classes in the equivalence relation induced by Equation~\eqref{EquationIdentificationLabelsCriticalPoints}. In particular, for any class $[\alpha,0]$, denoting by $\beta$ the letter with $\pi_t(\beta)=\pi_b(\alpha)$, Equation~\eqref{EquationDefinitionThurstonMap} becomes explicitly 
$$
v'([\alpha,-1])=
v'([\alpha,q^{(r)}_\alpha+h(\beta,r)-h(\alpha,r)-1]):=
f_\tau^{-1}\big(v([\beta,0])\big).
$$
\end{remark}

\begin{proposition}
\label{PropositionDefinitionThurstonMap}
For any $\gamma$-configuration $\cV$ the set $\cV'$ defined above by Equation~\eqref{EquationDefinitionThurstonMap} is a $\gamma$-configuration. In other words we have a well defined map 
$$
\cT_\gamma:\cO(\gamma)\to\cO(\gamma)
\quad
\textrm{ ; }
\quad
\cV\mapsto\cV'
$$
called the \emph{Thurston map}.
\end{proposition}

The map $\cT_\gamma:\cO(\gamma)\to\cO(\gamma)$ depends of course on the specific full family 
$
(f_\tau)_{\tau\in\Delta^\cA}
$.

\begin{proof}
Let $\tau=\tau(\cV)$ be defined by Equation~\eqref{EquationTauFunctionOfConfiguration} and let 
$
f_\tau=F(\tau)
$ 
be the corresponding GIET. Observe first that 
$
v([\alpha,1])=u^b_\alpha(f_\tau)
$ 
for any $\alpha\in\cA$, according to Point (2) in Definition~\ref{DefinitionFullFamily}, therefore 
\begin{equation}
\label{EquationPropositionDefinitionThurstonMap}
v'([\alpha,0])=f_\tau^{-1}\big(u^b_\alpha(f_\tau)\big)=u^t_\alpha(f_\tau).
\end{equation}
Since $f_\tau$ and 
$
T_\gamma=T(\pi,\lambda^{(\gamma)})
$ 
have the same combinatorial datum $\pi$, then the equality above implies $v'([\alpha_0,0])=0$ for the letter with $\pi_t(\alpha_0)=1$. Moreover for all other critical points we have
$$
v'([\alpha,0])< v'([\beta,0])  
\quad
\Leftrightarrow
\quad
\underline{v([\alpha,0])}< \underline{v([\beta,0])}.
$$

Now consider any pair of different points $v([\alpha,i])$ and $v([\beta,j])$ in the configuration $\cV$ and assume without loss of generality that $v([\alpha,i])<v([\beta,j])$, which is equivalent to the condition 
$
\underline{v([\alpha,i])}<\underline{v([\beta,j])}
$. 
The proposition follows proving that $v'([\alpha,i])<v'([\beta,j])$. Below, we consider separately two cases. Before proceeding, recall that if $f:[0,1)\to[0,1)$ is a GIET then for any $\alpha\in\cA$ we denote $I^t_\alpha(f)$ and $I^b_\alpha(f)$ the intervals such that the restricted map in Equation~\eqref{EquationPiecewiseHomeomorphism} is an orientation preserving homeomorphism. In particular, for $T_\gamma=T(\pi,\lambda^{(\gamma)})$ and for any $\alpha\in\cA$ we have 
\begin{align*}
&
\min I^t_\alpha(T_\gamma)=\underline{v([\alpha,0])}
\quad,\quad
\sup I^t_\alpha(T_\gamma)=\underline{v([\alpha,0])}+\lambda^{(\gamma)}_\alpha
\\
&
\min I^b_\alpha(T_\gamma)=\underline{v([\alpha,1])}
\quad,\quad
\sup I^b_\alpha(T_\gamma)=\underline{v([\alpha,1])}+\lambda^{(\gamma)}_\alpha.
\end{align*}

Assume first that 
$
\underline{v([\alpha,i])}
$ 
and 
$
\underline{v([\beta,j])}
$ 
are in the same continuity interval of 
$
T_\gamma
$, 
that is there is a letter $\chi\in\cA$ such that  
$$
\min I^t_\chi(T_\gamma)
=
\underline{v([\chi,0])}
\leq
\underline{v([\alpha,i])}<\underline{v([\beta,j])}
<
\sup I^t_\chi(T_\gamma).
$$ 
In this case we have
\begin{align*}
\min I^b_\chi(T_\gamma)=
\underline{v([\chi,1])}
\leq
&
T_\gamma\big(\underline{v([\alpha,i])}\big)=
\underline{v([\alpha,i+1])}<
\\
&
\underline{v([\beta,j+1])}=
T_\gamma\big(\underline{v([\beta,j])}\big)
<
\sup I^b_\chi(T_\gamma).
\end{align*}
All the critical values of $f_\tau$ are points of the configuration $\cV$, moreover the geometrical order of the points in $\cV$ is the same as in $\cV^{(\gamma)}$, hence we have 
$$
\min I^b_\chi(f_\tau)
=
v([\chi,1])
\leq
v([\alpha,i+1])<v([\beta,j+1])
<
\sup I^b_\chi(f_\tau).
$$ 
It follows that $v([\alpha,i+1])$ and $v([\beta,j+1])$ belong to the same (right-open) continuity interval of $f^{-1}_\tau$, and thus 
$$
v'([\alpha,i])=
f_\tau^{-1}\big(v([\alpha,i+1])\big) 
<
f_\tau^{-1}\big(v([\beta,j+1])\big)
=
v'([\beta,j]).
$$

Assume now that 
$
\underline{v([\alpha,i])}
$ 
and 
$
\underline{v([\beta,j])}
$ 
belong to two different continuity intervals of $T_\gamma$, labelled respectively by letters $\chi_1$ and $\chi_2$ with $\pi_t(\chi_2)\geq\pi_t(\chi_1)+1$, that is  
\begin{align*}
\min I^t_{\chi_1}(T_\gamma)
\leq
&
\underline{v([\alpha,i])}<
\sup I^t_{\chi_1}(T_\gamma)
\leq
\min I^t_{\chi_2}(T_\gamma)
\leq
\underline{v([\beta,j])}<
\sup I^t_{\chi_2}(T_\gamma,0).
\end{align*}
In this case we have
\begin{align*}
&
\min I^b_{\chi_1}(T_\gamma)=
\underline{v([\chi_1,1])}\leq
\underline{v([\alpha,i+1])}<
\sup I^b_{\chi_1}(T_\gamma)
\\
&
\min I^b_{\chi_2}(T_\gamma)=
\underline{v([\chi_2,1])}\leq
\underline{v([\beta,j+1])}<
\sup I^b_{\chi_2}(T_\gamma).
\end{align*}
As in the previous case, since all the critical values of $f_\tau$ are points of the configuration $\cV$, and since the geometrical order of the point in $\cV$ is the same as in $\cV^{(\gamma)}$, we have 
\begin{align*}
&
\min I^b_{\chi_1}(f_\tau)=
v([\chi_1,1])\leq
v([\alpha,i+1])<
\sup I^b_{\chi_1}(f_\tau)
\\
&
\min I^b_{\chi_2}(f_\tau)=
v([\chi_2,1])\leq
v([\beta,j+1])<
\sup I^b_{\chi_2}(f_\tau).
\end{align*}
Recalling Equation~\eqref{EquationDefinitionThurstonMap} and arguing as in the previous case we get
\begin{align*}
&
\min I^t_{\chi_1}(f_\tau)=
v'([\chi_1,0])\leq
v'([\alpha,i])<
\sup I^t_{\chi_1}(f_\tau)
\\
&
\min I^t_{\chi_2}(f_\tau)=
v'([\chi_2,0])\leq
v'([\beta,j])<
\sup I^t_{\chi_2}(f_\tau).
\end{align*}
The proof is completed observing that 
$
\sup I^t_{\chi_1}(f_\tau)\leq \sup I^t_{\chi_2}(f_\tau)
$, 
because we proved yet that the critical points of $f_\tau$ have the same geometrical order as those of $T_\gamma$.
\end{proof}

\begin{lemma}
\label{LemmaContinuityThurstonMap}
The Thurston map 
$
\cT_\gamma:\cO(\gamma)\to\cO(\gamma)
$ 
defined by Equation~\eqref{EquationDefinitionThurstonMap} is continuous.
\end{lemma}

\begin{proof}
For a configuration $\cV\in\cO(\gamma)$ set 
$
\cV':=\widehat{\cT}_\gamma(\cV)
$, 
then for $[\alpha,i]\in\cI(\gamma)$ let $v([\alpha,i])$ and $v'([\alpha,i])$ be the points in $\cV$ and $\cV'$ respectively. We prove that $v'([\alpha,i])$ depends continuously on $\cV$ for any 
$
[\alpha,i]\in\cI(\gamma)
$. 
Observe first that this is true for any point $v'([\alpha,0])$ with $\alpha\in\cA$, according to Equation~\eqref{EquationPropositionDefinitionThurstonMap} and the continuity of the composition 
$
\cV\mapsto\tau(\cV)\mapsto f_{\tau(\cV)}
$. 
In order to treat all other points, let 
$
\epsilon=\epsilon(\cV)>0
$ 
be small enough so that the compact set
$$
K:=[0,1)\setminus
\bigsqcup_{\alpha\in\cA}B\big(v([\alpha,1]),\epsilon\big)
$$
contains all the points $v([\alpha,i])$ of $\cV$ with $i\not=1$. Let 
$
\Omega=\Omega(\cV,\epsilon)\subset\cO(\gamma)
$ 
be an open set such that for any $\gamma$-configuration $\widetilde{\cV}\in\Omega$, denoting by 
$
\widetilde{v}([\alpha,i])
$ 
with $[\alpha,i]\in\cI(\gamma)$ the points in $\widetilde{\cV}$, we have
$$
\widetilde{v}([\alpha,1])\not\in K
\quad
\textrm{ and }
\quad
\widetilde{v}([\alpha,i])\in K
\quad
\textrm{ for }
\quad
i\not=1.
$$
Set 
$
\widetilde{\cV}':=\cT_\gamma(\widetilde{\cV})
$ 
and denote by $\widetilde{v}'([\alpha,i])$ the points in $\widetilde{\cV}'$. Finally let  
$
\tau:=\tau(\cV)
$ 
and 
$
\widetilde{\tau}:=\tau(\widetilde{\cV})
$ 
be the elements in $\Delta^\cA$ given by Equation~\eqref{EquationTauFunctionOfConfiguration}. For any $[\alpha,i]\in\cI(\gamma)$ with $i\not=0$ we have 
\begin{align*}
&
|v'([\alpha,i])-\widetilde{v}'([\alpha,i)]|
=
\big|
f_\tau^{-1}\big(v([\alpha,i+1])\big)-
f_{\widetilde{\tau}}^{-1}\big(\widetilde{v}([\alpha,i+1])\big)
\big|\leq
\\
&
\big|
f_\tau^{-1}\big(v([\alpha,i+1])\big)-
f_{\tau}^{-1}\big(\widetilde{v}([\alpha,i+1])\big)
\big|
+
\big|
f_\tau^{-1}\big(\widetilde{v}([\alpha,i+1])\big)-
f_{\widetilde{\tau}}^{-1}\big(\widetilde{v}([\alpha,i+1])\big)
\big|.
\end{align*}
The fist term in the second line is small because 
$
v([\alpha,i+1])\big)
$ 
and 
$
\widetilde{v}([\alpha,i+1])
$ 
are two nearby points in a continuity interval of $f_{\tau}^{-1}$. The second term in the second line is small because the point 
$
\widetilde{v}([\alpha,i+1])
$ 
is in an interval where both $f_\tau^{-1}$ and $f_{\widetilde{\tau}}^{-1}$ acts continuously and moreover the map $\tau\mapsto f_{\tau}^{-1}$ is continuous by Point (1) in Definition~\ref{DefinitionFullFamily}. 
\end{proof}

\subsection{Fixed points for the Thurston map} 
\label{SectionFixedPointsThurstonMap}

In \S~\ref{SectionExistenceFixedPointThurstonMap} below we prove the existence of a $\gamma$-configuration which is fixed by the Thurston map 
$
\cT_\gamma:\cO(\gamma)\to\cO(\gamma)
$. 
Lemma~\ref{LemmaFixedPointsThurstonAndRauzyPath} below explains why fixed points of the Thurston map are relevant.

\begin{lemma}
\label{LemmaFixedPointsThurstonAndRauzyPath}
Let $\cV\in\cO(\gamma)$ be a $\gamma$-configuration such that $\cT_\gamma(\cV)=\cV$ and let $\tau=\tau(\cV)$ be given by Equation~\eqref{EquationTauFunctionOfConfiguration}. Then $f_\tau$ has the same Rauzy path as $T_\gamma$ up to time $r$, that is 
$$
\gamma(f_\tau,r)=\gamma=\gamma(T_\gamma,r).
$$
\end{lemma}

\begin{proof}
It is enough to prove 
$
\cP(f_\tau,r)\sim\cP(T_\gamma,r)
$, 
according to Proposition~\ref{PropositionPartitionDeterminesRauzyPath}. Observe that if 
$
\cT_\gamma(\cV)=\cV
$ 
then Equation~\eqref{EquationDefinitionThurstonMap} implies 
$
v([\alpha,i+1])=f_\tau\big(v([\alpha,i])\big)
$ 
for any $[\alpha,i]\in\cI(\gamma)$, where we recall that the identifications between labels in $\cA\times\ZZ$ are given in Equation~\eqref{EquationIdentificationLabelsCriticalPoints}. It follows that the points in $\cV$ form a closed orbit of $f_\tau$, and their dynamical order is the same as the cyclic dynamical order of the points in $\cV^{(\gamma)}$ under the action of $T_\gamma$. By assumption $\cV$ contains the critical values $u^b_\alpha(f_\tau)$ of $f_\tau$, and thus also the critical points 
$
u^t_\alpha(f_\tau)=f_\tau^{-1}\big(u^b_\alpha(f_\tau)\big)
$, 
where $\alpha\in\cA$. Thereofre $\cV$ equals the set of points 
$$
f_\tau^i\big(u^t_\alpha(f_\tau)\big)
\quad,\quad
\alpha\in\cA
\quad,\quad
-h(\alpha,r)\leq i\leq q^{(r)}_\alpha-h(\alpha,r)-1,
$$
where the integers $h(\alpha,r)$ are defined right before Equation~\eqref{EquationIdentificationLabelsCriticalPoints}. Moreover these points have the same geometrical order as the points of the standard configuration $\cV^{(\gamma)}$, which are
$$
T_\gamma^i\big(u^t_\alpha(T_\gamma)\big)
\quad,\quad
\alpha\in\cA
\quad,\quad
-h(\alpha,r)\leq i\leq q^{(r)}_\alpha-h(\alpha,r)-1.
$$
These two families of points are the endpoints of the intervals in the partitions $\cP(f_\tau,r)$ and 
$\cP(T_\gamma,r)$ respectively, thus the Lemma follows.
\end{proof}

\section{Existence of fixed points for the Thurston map}
\label{SectionExistenceFixedPointThurstonMap}

In this section we prove Theorem~\ref{TheoremRealizationRauzyPath} via Proposition~\ref{PropositionCyclicBehaviour} below. We use the notation introduced in \S~\ref{SectionThurstonMap}.  In particular, we fix an admissible combinatorial datum $\pi$ and we assume that the Rauzy class $\cR$ of $\pi$ contains a cyclic combinatorial datum $\pi^{(\ast)}$. Then we fix a finite Rauzy path 
$
\gamma:\pi\to\pi^{(\ast)}
$.

\subsection{The boundary of the configuration space}
\label{SectionBoundaryConfigurationSpace}

Recall that configuration space $\cO(\gamma)$ is identified with the open simplex 
$
\{\cV^{(\gamma)}\}\times\Delta^N
$. 
A \emph{degenerate $\gamma$-configuration} is a set $\cV$ of points $v([\alpha,i])\in[0,1]$ labeled by classes 
$
[\alpha,i]\in\cI(\gamma)
$ 
such that there exist $\gamma$-configurations $\cV_n\in\cO(\gamma)$, where $n\in\NN$ and 
$
\cV_n=\{v_n([\alpha,i]);[\alpha,i]\in\cI(\gamma)\}
$,  
with 
$$
v([\alpha,i])=\lim_{n\to\infty}v_n([\alpha,i])
\quad
\textrm{ for any }
\quad
[\alpha,i]\in\cI(\gamma)
$$
We also require that degenerate configurations are not configurations in the standard sense, that is there must be either two sequences $v_n([\alpha,i])$ and $v_n([\beta,j])$ with 
$
[\alpha,i]\not=[\beta,j]
$ 
converging to the same limit for $n\to\infty$, or some sequence with $v_n([\alpha,i])\to1$ for $n\to\infty$. Such degenerate configurations are elements of the boundary $\partial\cO(\gamma)$ of the simplex $\cO(\gamma)$, whose faces are described combinatorially below.

\medskip

Let $[\alpha_m,i_m]\in\cI(\gamma)$ be the label such that 
$
\underline{v([\alpha_m,i_m])}
$ 
is the rightmost element in the standard configuration $\cV^{(\gamma)}$, that is 
$
\underline{v([\alpha_m,i_m])}=1-N^{-1}
$, 
or equivalently the label such that for any $[\alpha,i]\in\cI(\gamma)$ we have
$$
\underline{v([\alpha,i])}\leq\underline{v([\alpha_m,i_m])}.
$$
The set $\cI(\gamma)$ admits a \emph{geometrical order}, where for any $[\alpha,i]\in\cI(\gamma)$ with 
$
[\alpha,i]\not=[\alpha_m,i_m]
$ 
we define $[\alpha,i]^\ast$ as the element of $\cI(\gamma)$ such that  
$$
\underline{v([\alpha,i]^\ast)}=
\min\{\underline{v([\beta,j])}
\textrm{ ; }
\underline{v([\beta,j])}>\underline{v([\alpha,i])}\}
\quad
\Leftrightarrow
\quad
\underline{v([\alpha,i]^\ast)}=\underline{v([\alpha,i])}+\frac{1}{N}.
$$

First, define the \emph{maximal face} 
$
\partial^{[\alpha_m,i_m]}\cO(\gamma)
$ 
as the set of those degenerate $\gamma$-configurations 
$
\cV=\{v([\beta,j])\in[0,1],[\beta,j]\in\cI(\gamma)\}
$ 
with 
$$
v([\alpha_m,i_m])=1.
$$

For any $[\alpha,i]\in\cI(\gamma)$ with 
$
[\alpha,i]\not=[\alpha_m,i_m]
$,  
define the face $\partial^{[\alpha,i]}\cO(\gamma)$ as the set of degenerate $\gamma$-configurations 
$
\cV=\{v([\beta,j])\in[0,1],[\beta,j]\in\cI(\gamma)\}
$ 
such that 
$$
v([\alpha,i])=v([\alpha,i]^\ast).
$$

Finally define
$$
\overline{\cO(\gamma)}:=
\cO(\gamma)
\cup
\bigcup_{[\alpha,i]\in\cI(\gamma)}\partial^{[\alpha,i]}\cO(\gamma).
$$

Let 
$
\cT_\gamma:\cO(\gamma)\to\cO(\gamma)
$ 
be the Thurston map defined in \S~\ref{SectionDefinitionThurstonMap}. The main result of this section is  Proposition~\ref{PropositionCyclicBehaviour} below, which is used in the next \S~\ref{SectionProofTheoremRealizationRauzyPath} to prove Theorem~\ref{TheoremRealizationRauzyPath}. The proof of the Proposition is the subject of the remaining part of the section.

\begin{proposition}
\label{PropositionCyclicBehaviour}
There exists a map 
$
\widehat{\cT}_\gamma:\overline{\cO(\gamma)}\to\overline{\cO(\gamma)}
$ 
satisfying the properties below.
\begin{enumerate}
\item
We have
$
\widehat{\cT}_\gamma(\cV)=\cT_\gamma(\cV)
$ 
for any $\cV\in\cO(\gamma)$, that is $\widehat{\cT}_\gamma$ extends the map $\cT_\gamma$ to 
$
\overline{\cO(\gamma)}
$.
\item
The extended map $\widehat{\cT}_\gamma$ is continuous on 
$
\overline{\cO(\gamma)}
$.
\item
The extended map $\widehat{\cT}_\gamma$ leaves $\partial\cO(\gamma)$ invariant. Moreover for any 
$
[\alpha,i]\in\cI(\gamma)
$ 
we have 
$$
\widehat{\cT}_\gamma
\big(\partial^{[\alpha,i]}\cO(\gamma)\big)
\subset 
\partial^{[\alpha,i-1]}\cO(\gamma). 
$$
\end{enumerate}
\end{proposition}

\begin{proof}
The proof of the Proposition is the content of \S~\ref{SectionExtensionToTheBoundary}, \S~\ref{SectionContinuityExtendedThurstonMap} and \S~\ref{SectionCyclicBehavior} below. 
\end{proof}

\subsection{Proof of Theorem~\ref{TheoremRealizationRauzyPath}}
\label{SectionProofTheoremRealizationRauzyPath}

According to Point (2) of Proposition~\ref{PropositionCyclicBehaviour} and to Brower's fixed point Theorem, there exists a configuration $\cV\in\overline{\cO(\gamma)}$ such that 
$
\widehat{\cT}_\gamma(\cV)=\cV
$. 
Recall that $\pi^{(\ast)}$ is a cyclic combinatorial datum, that is $\sigma(\pi^{(\ast)})$ is a cyclic permutation of $\{1,\dots,d\}$ with maximal order $d$. Hence the map 
$
[(\alpha,i)]\mapsto[(\alpha,i-1)]
$ 
is a cyclic permutation of all classes in $\cI(\gamma)$, with maximal order $N$. Therefore, if 
$
\cV\in\partial\cO(\gamma)
$, 
that is such fixed point $\cV$ is a degenerate $\gamma$-configuration, then Point (3) of Proposition~\ref{PropositionCyclicBehaviour} implies 
$$
\cV\in
\bigcap_{[\alpha,i]\in\cI(\gamma)}\partial^{[\alpha,i]}\cO(\gamma),
$$
which is absurd, because the intersection of all faces of $\Delta^N$ is empty. We must have 
$\cV\in\cO(\gamma)$ and 
$
\cT_\gamma(\cV)=\widehat{\cT}_\gamma(\cV)=\cV
$, 
hence Theorem~\ref{TheoremRealizationRauzyPath} follows from Lemma~\ref{LemmaFixedPointsThurstonAndRauzyPath}.

\subsection{The extension to the boundary}
\label{SectionExtensionToTheBoundary}

Lemma~\ref{LemmaDefinitionExtensionThurstonMap} in this subsection proves Point (1) in Proposition~\ref{PropositionCyclicBehaviour}. Equation~\eqref{EquationTauFunctionOfConfiguration} defines a continuous extension 
$$
\tau(\cdot):\overline{\cO(\gamma)}\to\overline{\Delta^\cA}
\quad
\textrm{ ; }
\quad
\cV\mapsto\tau=\tau(\cV).
$$
Observe that for a degenerate configurations $\cV\in\partial\cO(\gamma)$ it is possible to have a non degenerate vector  
$
\tau(\cV)\in\Delta^\cA
$ 
and not necessarily 
$
\tau(\cV)\in\partial\Delta^\cA
$. 
Let $(f_\tau)_{\tau\in\Delta^\cA}$ be a full family of GIETs parametrized by 
$
\tau\mapsto F(\tau):=f_\tau
$ 
and consider the extension 
$$
\widehat{F}:
\overline{\Delta^\cA}\to\widehat{\cG}\big(\pi,[0,1)\big)
$$
as in Point (3) of Definition~\ref{DefinitionFullFamily}. Recall that if $\tau\in\Delta^\cA$, then 
$
\widehat{F}(\tau)=f_\tau\in\cG\big(\pi,[0,1)\big)
$. 
Otherwise if 
$
\tau\in\partial\Delta^\cA
$, 
then we have a degeneration 
$
D_\tau=\widehat{F}(\tau)
$, 
which we decompose as $D_\tau=G_{f_\tau}\cup S_\tau$, where $S_\tau$ is the singular part and 
$
f_\tau:[0,1)\to[0,1)
$ 
is a GIET whose combinatorial datum is a reduction of $\pi$. In both cases we say that $f_\tau$ is the regular part of $\widehat{F}(\tau)$. Composition Equation~\eqref{EquationTauFunctionOfConfiguration} with $\widehat{F}$ we get a continuous map
\begin{equation}
\label{EquationGietFunctionOfConfiguration(Extension)}
\overline{\cO(\gamma)}\to\widehat{\cG}(\pi,[0,1))
\quad
\textrm{ ; }
\quad
\cV\mapsto\widehat{F}\big(\tau(\cV)\big),
\end{equation}
where the distance between configurations $\widetilde{\cV}$ and $\cV$ is the restriction to $\overline{\Delta^N}$ of the distance on $\RR^N$ induced by the sup-norm $\|\cdot\|_\infty$ and where the distance on $\widehat{\cG}\big(\pi,[0,1)\big)$ is defined by Equation~\eqref{EquationDinstanceGIETsDegenerate}. In the following we will apply frequently Lemma~\ref{LemmaControlExpansionInverseBranches} below.

\begin{lemma}
\label{LemmaControlExpansionInverseBranches}
Let $\tau\in\overline{\Delta^\cA}$ and 
$
(\tau_n)_{n\in\NN}\subset\Delta^\cA
$ 
be a sequence with $\tau_n\to\tau$ as $n\to\infty$. Let $I_n$ be a sequence of open intervals with $|I_n|\to0$ as $n\to\infty$ such that there exists $\alpha\in\cA$ with 
$$
I_n\subset I^b_\alpha(f_{\tau_n})
\quad
\textrm{ for any }
\quad
n\in\NN.
$$
Then $f_{\tau_n}^{-1}(I_n)$ is an interval too and we have $|f_{\tau_n}^{-1}(I_n)|\to0$. 
\end{lemma}

\begin{proof}
Recall that  
$
\distance\big(f_{\tau_n},\widehat{F}(\tau)\big)\to0
$ 
as $n\to\infty$ in the distance defined by Equation~\eqref{EquationDinstanceGIETsDegenerate}. Observe that 
$
f_{\tau_n}^{-1}(I_n)\subset I^t_\alpha(f_{\tau_n})
$ 
for any $n$. Thus in particular the Lemma is obvious if 
$
|I^t_\alpha(f_{\tau_n})|\to0
$ 
as $n\to\infty$, that is $\tau\in\partial\Delta^\cA$ with $\tau_\alpha=0$. Otherwise there exists non-empty right-open intervals $I^t_\alpha$ and $I^b_\alpha$ with
$$
\closure\big(I^t_\alpha(f_{\tau_n})\big)\to \closure(I^t_\alpha)
\quad
\textrm{ and }
\quad
\closure\big(I^b_\alpha(f_{\tau_n})\big)\to \closure(I^b_\alpha)
$$
in the Hausdorff distance as $n\to\infty$, and the restriction 
$
f_\tau:I^t_\alpha \to I^b_\alpha
$ 
is a homeomorphism, where $f_\tau$ is the regular part of $\widehat{F}(\tau)$. For any $n\in\NN$ set $I_n=(c_n,d_n)$ and let 
$
a_n:=f_{\tau_n}^{-1}(c_n)
$ 
and 
$
b_n:=f_{\tau_n}^{-1}(d_n)
$. 
Modulo subsequences, let $a,b\in\closure(I^t_\alpha)$ with $a_n\to a$ and $b_n\to b$ as $n\to\infty$, where obviously $a\leq b$. If $a<b$ strictly then 
$$
\epsilon:=\lim_{x\to b-}f_\tau(x)-\lim_{x\to a+}f_\tau(x)>0
$$ 
also strictly. In this case, let $n_0$ be such that for $n\geq n_0$ we have 
$
0<d_n-c_n\leq \epsilon/50
$. 
Assume also that $|a_n-a|<\epsilon/100$ and $|b_n-b|<\epsilon/100$ for any such $n$. Then we have 
$$
\distance\big(f_{\tau_n},\widehat{F}(\tau)\big)
\geq
\max
\left\{
|c_n-\lim_{x\to a+}f_\tau(x)|,|d_n-\lim_{x\to b-}f_\tau(x)|
\right\}
\geq\frac{\epsilon}{100},
$$
for any $n$ big enough, which is absurd. The Lemma is proved.
\end{proof}

Fix 
$
\cV\in\overline{\cO(\gamma)}
$, 
that is a $\gamma$-configuration which can be degenerate or not. Consider a sequences of configurations $\cV_n\in\cO(\gamma)$, where $n\in\NN$ and 
$
\cV_n=\{v_n([\alpha,i]),[\alpha,i]\in\cI(\gamma)\}
$, 
such that $\cV_n\to\cV$, that is such that 
$
\lim_{n\to\infty}v_n([\alpha,i])= v([\alpha,i])
$ 
for any $[\alpha,i]\in\cI(\gamma)$. Define the set of points 
$
\cV'=\{v'([\alpha,i]);[\alpha,i]\in\cI(\gamma)\}\subset[0,1)
$ 
by 
\begin{equation}
\label{EquationDefinitionExtensionThurstonMap}
v'([\alpha,i]):=
\lim_{n\to\infty}f_{\tau(\cV_n)}^{-1}\big(v_n([\alpha,i+1])\big),
\end{equation}
where we recall Remark~\ref{RemarkIdentificationlabels}. According to Lemma~\ref{LemmaDefinitionExtensionThurstonMap} below, Equation~\eqref{EquationDefinitionExtensionThurstonMap} defines a map
$$
\widehat{\cT}_\gamma:\underline{\cO(\gamma)}\to\underline{\cO(\gamma)}
\quad
\textrm{ ; }
\quad
\cV\mapsto\widehat{\cT}_\gamma(\cV):=\cV'
$$
which extends the map 
$
\cT_\gamma:\cO(\gamma)\to\cO(\gamma)
$, 
called the \emph{extended Thurston map}.

\begin{lemma}
\label{LemmaDefinitionExtensionThurstonMap}
Fix 
$
\cV\in\overline{\cO(\gamma)}
$. 
\begin{enumerate}
\item
The limit in Equation~\eqref{EquationDefinitionExtensionThurstonMap} exists.
\item
The set $\cV'$ defined by Equation~\eqref{EquationDefinitionExtensionThurstonMap} does not depend on the choice of the sequence $\cV_n\in\cO(\gamma)$ with $\cV_n\to\cV$. 
\item
We have  
$
\cV'\in\overline{\cO(\gamma)}
$.
\item
If $\cV\in\cO(\gamma)$, that is $\cV$ is non degenerate, then $\cV'=\cT_\gamma(\cV)$.
\end{enumerate}
\end{lemma}

\begin{proof} 
Consider $\cV_n\in\cO(\gamma)$ with $\cV_n\to\cV$ as $n\to\infty$ and recall that, by continuity of the map in Equation~\eqref{EquationGietFunctionOfConfiguration(Extension)}, we have 
$
f_{\tau(\cV_n)}\to \widehat{F}\big(\tau(\cV)\big)
$ 
in terms of the distance defined by Equation~\eqref{EquationDinstanceGIETsDegenerate}. For $i=0$, Equation~\eqref{EquationDefinitionExtensionThurstonMap} becomes
$$
\lim_{n\to+\infty}
f_{\tau(\cV_n)}^{-1}\big(v_n([\alpha,1])\big)=
\lim_{n\to+\infty}
u^t_\alpha(f_{\tau(\cV_n)}),
$$
and the limit above exists and is independent from the choice of $\cV_n$ with $\cV_n\to\cV$, so that Point (1) and (2) are proved for elements $v'([\alpha,0])$ of $\cV'$. For $i\not=0$, since all $\cV_n$ are $\gamma$-configuration, then there exists a letter $\chi\in\cA$ such that 
$$
u^b_\chi(f_{\tau(\cV_n)})
=
\min I^b_\chi(f_{\tau(\cV_n)})
<
v_n([\alpha,i+1])
<
\sup I^b_\chi(f_{\tau(\cV_n)})
$$
for any $n$. The argument for $i=0$ implies that there exist points $a=a(\chi)$ and $b=b(\chi)$ in the closed interval $[0,1]$, which do not depend on the choice of $\cV_n\to\cV$, such that 
$$
u^b_\chi(f_{\tau(\cV_n)})\to a(\chi)
\quad
\textrm{ and }
\quad
\sup I^b_\chi(f_{\tau(\cV_n)})\to b(\chi)
$$ 
as $n\to\infty$. We have obviously $a\leq b$, where $a=b$ holds when the intervals 
$
I^{t/b}_\chi(f_{\tau(\cV_n)})
$ 
shrink to a point as $n\to\infty$. For any $n$, consider the two open sub-intervals of $I^b_\chi(f_{\tau(\cV_n)})$ given by
$$
I_n:=
\big(
\min I^b_\chi(f_{\tau(\cV_n)}),v_n([\alpha,i+1])
\big)
\quad
\textrm{ and }
\quad
I'_n:=
\big(
v_n([\alpha,i+1]),\sup I^b_\chi(f_{\tau(\cV_n)})
\big).
$$ 
If either $|I_n|\to0$ or $|I'_n|\to0$ as $n\to\infty$, then according to Lemma~\ref{LemmaControlExpansionInverseBranches} we have either  
$
|f_{\tau(\cV_n)}^{-1}(I_n)|\to0
$ 
or 
$
|f_{\tau(\cV_n)}^{-1}(I'_n)|\to0
$ 
respectively, and thus 
$
f_{\tau(\cV_n)}^{-1}\big(v_n([\alpha,i+1])\big)
$ 
converges either to $a(\chi)$ or to $b(\chi)$ respectively. In this case too Point (1) and (2) holds for $v'([\alpha,i])$. If both $\inf_{n\in\NN}|I_n|>0$ and  $\inf_{n\in\NN}|I'_n|>0$ then the letter $\chi$ corresponds to a non-degenerate branch of the regular part $f_{\tau(\cV)}$ of 
$
\widehat{F}\big(\tau(\cV)\big)
$, 
and there exist right open intervals 
$
I^{t/b}_\chi\subset[0,1)
$ 
such that the restriction 
$
f_{\tau(\cV)}^{-1}:I^b_\chi\to I^t_\chi
$ 
is an homeomorphism. Moreover we have 
$
v([\alpha,i+1])\in \interior(I^b_\chi)
$, 
and thus also 
$
v_n([\alpha,i+1])\in \interior(I^b_\chi)
$ 
for $n$ big enough. We have 
\begin{align*}
&
\left|
f_{\tau(\cV_n)}^{-1}\big(v_n([\alpha,i+1])\big)
-
f_{\tau(\cV)}^{-1}\big(v([\alpha,i+1])\big)
\right|
\leq
\\
&
\left|
f_{\tau(\cV_n)}^{-1}\big(v_n([\alpha,i+1])\big)
-
f_{\tau(\cV_n)}^{-1}\big(v([\alpha,i+1])\big)
\right|
+
\left|
f_{\tau(\cV_n)}^{-1}\big(v([\alpha,i+1])\big)
-
f_{\tau(\cV)}^{-1}\big(v([\alpha,i+1])\big)
\right|,
\end{align*}
therefore 
$
f_{\tau(\cV_n)}^{-1}\big(v_n([\alpha,i+1])\big)
\to
f_{\tau(\cV)}^{-1}\big(v([\alpha,i+1])\big), 
$ 
indeed the first term in the second line converges to $0$ by Lemma~\ref{LemmaControlExpansionInverseBranches}, while the second converges to $0$ because 
$
f_{\tau(\cV_n)}\to \widehat{F}\big(\tau(\cV)\big)
$. 
This concludes the proof of Point (1) and (2) in the Lemma. Point (3) follows observing that
$
\cV'=\lim_{n\to\infty}\cV'_n
$, 
where $\cV'_n=\cT_\gamma(\cV_n)$, and the latter is a sequence of $\gamma$-configurations according to Proposition~\ref{PropositionDefinitionThurstonMap}. Finally, Point (4) just corresponds to the continuity of the map 
$
\cT_\gamma:\cO(\gamma)\to\cO(\gamma)
$, 
proved in Lemma~\ref{LemmaContinuityThurstonMap}. The proof of the Lemma is complete.
\end{proof}

\subsection{Continuity of the extended map}
\label{SectionContinuityExtendedThurstonMap}

Lemma~\ref{LemmaContinuityExtendedThurstonMap} in this subsection proves Point (2) in Proposition~\ref{PropositionCyclicBehaviour}. 

\begin{lemma}
\label{LemmaContinuityExtendedThurstonMap}
The extended Thurston map 
$
\widehat{\cT}_\gamma:\overline{\cO(\gamma)}\to\overline{\cO(\gamma)}
$ 
defined by Equation~\eqref{EquationDefinitionExtensionThurstonMap} is continuous.
\end{lemma}

\begin{proof}
Consider configurations $\cV,\widetilde{\cV}$ in $\overline{\cO(\gamma)}$, which can be degenerate or not. Let $\cV_n$ and $\widetilde{\cV}_n$ be non-degenerate configurations in $\cO(\gamma)$ with 
$\cV_n\to\cV$ and 
$
\widetilde{\cV}_n\to\widetilde{\cV}
$ 
as $n\to\infty$, so that 
$$
\cV':=\widehat{\cT}_\gamma(\cV)=\lim_{n\to\infty}\cV'_n
\quad
\textrm{ and }
\quad
\widetilde{\cV}':=\widehat{\cT}_\gamma(\widetilde{\cV})=\lim_{n\to\infty}\widetilde{\cV}'_n,
$$
where $\cV'_n:=\cT_\gamma(\cV_n)$ and 
$
\widetilde{\cV}'_n:=\cT_\gamma(\widetilde{\cV}_n)
$ 
for any $n\in\NN$, according to the definition of $\widehat{\cT}_\gamma$ in Equation~\eqref{EquationDefinitionExtensionThurstonMap}. For any 
$[\alpha,i]\in\cI(\gamma)$ denote $v'([\alpha,i])$ and $\widetilde{v}'([\alpha,i])$ the corresponding labeled points respectively in $\cV'$ and $\widetilde{\cV}'$. Similarly, denote $v_n([\alpha,i])$, $v'_n([\alpha,i])$, 
$
\widetilde{v}_n([\alpha,i])
$ 
and 
$
\widetilde{v}'_n([\alpha,i])
$ 
the corresponding points respectively in $\cV_n$, $\cV'_n$, 
$
\widetilde{\cV}_n
$ 
and $\widetilde{\cV}_n'$. Fix $\epsilon>0$ and $[\alpha,i]\in\cI(\gamma)$. We prove that if 
$\widetilde{\cV}$ is close enough to $\cV$, then we have 
$
|v'([\alpha,i])-\widetilde{v}'([\alpha,i])|\leq\epsilon
$, 
and in order to do so it is enough to prove that if $n$ is big enough then we have  
\begin{equation}
\label{Equation(1)LemmaContinuityExtendedThurstonMap}
\left|
v'([\alpha,i])-\widetilde{v}'_n([\alpha,i])
\right|
<\epsilon.
\end{equation}

Observe that if 
$
|\widetilde{\cV}-\cV|<\delta
$ 
for some $\delta>0$ then also 
$
|\widetilde{\cV}_n-\cV_n|<\delta
$ 
for any $n$ big enough. Moreover the map 
$
\cV\mapsto \widehat{F}\big(\tau(\cV)\big)
$ 
in Equation~\eqref{EquationGietFunctionOfConfiguration(Extension)} is continuous. It follows that if $\delta$ is small enough and $n$ is big enough we have 
\begin{equation}
\label{Equation(2)LemmaContinuityExtendedThurstonMap}
\distance\left(
f_{\tau(\widetilde{\cV}_n)},
\widehat{F}\big(\tau(\cV)\big)
\right)<\epsilon.
\end{equation}

For the labels $[\alpha,0]$ with $\alpha\in\cA$, Equation~\eqref{Equation(1)LemmaContinuityExtendedThurstonMap} follows easily from Equation~\eqref{Equation(2)LemmaContinuityExtendedThurstonMap} observing that for any $n$ we have
\begin{align*}
&
v'_n([\alpha,0])=
f_{\tau(\cV_n)}^{-1}\big(v_n([\alpha,1])\big)=
u^t_\alpha(f_{\tau(\cV_n)})\to v'([\alpha,0])
\\
&
\widetilde{v}'_n([\alpha,0])=
f_{\tau(\widetilde{\cV}_n)}^{-1}\big(\widetilde{v}_n([\alpha,1])\big)=
u^t_\alpha(f_{\tau(\widetilde{\cV}_n)}).
\end{align*}

Now consider $[\alpha,i]\not=[\alpha,0]$ and let $\chi\in\cA$ be the letter such that for any $n$ we have 
\begin{align*}
&
u^b_\chi(f_{\tau(\cV_n)})=
\min I^b_\chi(f_{\tau(\cV_n)})
<
v_n([\alpha,i+1])
<
\sup I^b_\chi(f_{\tau(\cV_n)})
\\
&
u^b_\chi(f_{\tau(\widetilde{\cV}_n)})=
\min I^b_\chi(f_{\tau(\widetilde{\cV}_n)})
<
\widetilde{v}_n([\alpha,i+1])
<
\sup I^b_\chi(f_{\tau(\widetilde{\cV}_n)}),
\end{align*}
so that 
$
v'_n([\alpha,i])\in I^t_\chi(f_{\tau(\cV_n)})
$ 
and 
$
\widetilde{v}'_n([\alpha,i])\in I^t_\chi(f_{\tau(\widetilde{\cV}_n)})
$ 
by Equation~\eqref{EquationDefinitionThurstonMap}. Set 
$$
a=a(\chi):=\lim_{n\to\infty}u^t_\chi(f_{\tau(\cV_n)})
\quad
\textrm{ and }
\quad
b=b(\chi):=\lim_{n\to\infty}\left(\sup I^t_\chi(f_{\tau(\cV_n)})\right).
$$
According to Part (1) of Lemma~\ref{LemmaDefinitionExtensionThurstonMap}, we have 
$$
\lim_{n\to\infty}
f_{\tau(\cV_n)}^{-1}\big(v_n([\alpha,i+1])\big)=
v'([\alpha,i])\in [a,b].
$$
Moreover, if $\widetilde{\cV}$ is close to $\cV$ and $n$ is big enough, then Equation~\eqref{Equation(2)LemmaContinuityExtendedThurstonMap} implies
$$
\left|
a-\min I^t_\chi(f_{\tau(\widetilde{\cV}_n)})
\right|
<\epsilon
\quad
\textrm{ and }
\quad
\left|
\sup I^t_\chi(f_{\tau(\widetilde{\cV}_n)})-b
\right|
<\epsilon.
$$

If $a=v'([\alpha,i])=b$, then  
$
|\widetilde{v}'_n([\alpha,i])-v'([\alpha,i])|<\epsilon
$, 
so that Equation~\eqref{Equation(1)LemmaContinuityExtendedThurstonMap} is satisfied. If otherwise $a<b$ strictly, then the letter $\chi$ corresponds to a non-degenerate branch of the regular part $f_{\tau(\cV)}$ of 
$
\widehat{F}\big(\tau(\cV)\big)
$ 
and there exists a right-open interval 
$
I^b_\chi\subset[0,1]
$ 
such that the restriction 
$
f_{\tau(\cV)}^{-1}:I^b_\chi\to[a,b)
$ 
is an homeomorphism. By continuity of the map 
$
\cV\mapsto \widehat{F}\big(\tau(\cV)\big)
$ 
the letter $\chi$ also corresponds to a non-degenerate branch of the regular part $f_{\tau(\widetilde{\cV})}$ of 
$
\widehat{F}\big(\tau(\widetilde{\cV})\big)
$, 
and there exists right-open intervals 
$
\widetilde{I}^{t/b}_\chi\subset[0,1)
$ 
such that the restriction 
$
f_{\tau(\widetilde{\cV})}^{-1}:\widetilde{I}^{b}_\chi\to\widetilde{I}^{t}_\chi
$ 
is an homeomorphism onto its image. As in the proof of Part (1) of Lemma~\ref{LemmaDefinitionExtensionThurstonMap}, consider the two open sub-intervals of $I^b_\chi(f_{\tau(\cV_n)})$ given by
$$
I_n:=
\big(
\min I^b_\chi(f_{\tau(\cV_n)}),v_n([\alpha,i+1])
\big)
\quad
\textrm{ and }
\quad
I'_n:=
\big(
v_n([\alpha,i+1]),\sup I^b_\chi(f_{\tau(\cV_n)})
\big).
$$ 
If both $\sup_{n\in\NN}|I_n|>0$ and $\sup_{n\in\NN}|I'_n|>0$, then 
$
v([\alpha,i+1])=\lim_{n\to\infty}v([\alpha,i+1])
$ 
is an interior point of $I^b_\chi$. Moreover, if $\widetilde{\cV}$ is close enough to $\cV$, then also   
$
\widetilde{v}([\alpha,i+1])
$ 
is an interior point of $I^b_\chi$, and thus 
$
\widetilde{v}_n([\alpha,i+1])
$ 
too, for any $n$ big enough. Therefore Equation~\eqref{Equation(1)LemmaContinuityExtendedThurstonMap} follows with the same argument as in the second part of the proof of Lemma~\ref{LemmaContinuityThurstonMap}. Otherwise assume $|I_n|\to0$ as $n\to\infty$, where the argument for the case $|I'_n|\to0$ is the same and is left to the reader. We have 
$
|f_{\tau(\cV_n)}^{-1}(I_n)|\to0
$ 
and thus Lemma~\ref{LemmaControlExpansionInverseBranches} implies 
$$
v'_n([\alpha,i])=
f_{\tau(\cV_n)}^{-1}\big(v_n([\alpha,i+1])\big)\to a.
$$
Therefore Equation~\eqref{Equation(1)LemmaContinuityExtendedThurstonMap} holds for  
$
\widetilde{v}'_n([\alpha,i])
=
f_{\tau(\widetilde{\cV}_n)}^{-1}\big(\widetilde{v}_n([\alpha,i+1])\big)
$ 
because Equation~\eqref{Equation(2)LemmaContinuityExtendedThurstonMap} implies 
$$
\left|
\left(
f_{\tau(\widetilde{\cV}_n)}^{-1}
\big(
\widetilde{v}_n([\alpha,i+1])
\big),
\widetilde{v}_n([\alpha,i+1])
\right)
-
\left(
a,v([\alpha,i+1])
\right)
\right|_{\RR^2}\leq \epsilon,
$$
where $|p_1-p_2|_{\RR^2}$ denotes the distance between points $p_1,p_2$ in $\RR^2$. The Lemma is proved.
\end{proof}

\subsection{Cyclic behavior}
\label{SectionCyclicBehavior}

Lemma~\ref{LemmaCyclicBehavior} in this subsection proves Point (3) in Proposition~\ref{PropositionCyclicBehaviour}.

\begin{lemma}
\label{LemmaCyclicBehavior}
For any $[\alpha,i]\in\cI(\gamma)$ we have 
$$
\widehat{\cT}_\gamma\big(\partial^{[\alpha,i]}\cO(\gamma)\big)
\subset
\partial^{[\alpha,i-1]}\cO(\gamma).
$$
\end{lemma}

\begin{proof}
Consider a degenerate configuration 
$
\cV\in\partial\cO(\gamma)
$ 
and its image $\cV':=\widehat{\cT}_\gamma(\cV)$. Then let $\cV_n\in\cO(\gamma)$ be  a sequence with 
$\cV_n\to\cV$ as $n\to\infty$ and let $\cV'_n:=\cT_\gamma(\cV_n)$ be its image, so that 
$
\cV'=\lim_{n\to\infty}\cV'_n
$, 
according to the definition of the map 
$
\widehat{\cT}_\gamma
$ 
in Equation~\eqref{EquationDefinitionExtensionThurstonMap}. For $[\beta,j]\in\cI(\gamma)$ denote by 
$v([\beta,j])$, $v'([\beta,j])$, $v_n([\beta,j])$ and $v'_n([\beta,j])$ the corresponding labelled points respectively in the configurations $\cV$, $\cV'$, $\cV_n$ and $\cV'_n$. 

Fix $[\alpha,i]\in\cI(\gamma)$ and consider  
$
\cV\in\partial^{[\alpha,i]}\cO(\gamma)
$, 
that is a degenerate configuration with 
$
v([\alpha,i])=v([\alpha,i]^\ast)
$. 
The Lemma follows proving that for the image  configuration 
$
\cV'=\widehat{\cT}_\gamma(\cV)
$ 
we have 
$
v'([\alpha,i-1])=v'([\alpha,i-1]^\ast)
$, 
that is
\begin{equation}
\label{EquartionLemmaCyclicBehavior}
\big|
v'_n([\alpha,i-1]^\ast)
-
v'_n([\alpha,i-1])
\big|\to0
\quad
\textrm{ as }
\quad
n\to\infty.
\end{equation}
We consider separately the cases below. 

\smallskip

\emph{Case (1): leftmost point of $\cV$.} Let $[\alpha,i]:=[\alpha_0,0]$, where $\alpha_0$ is the letter with 
$\pi_t(\alpha_0)=1$. Recall that in this case  
$
v([\alpha_0,0])=v'([\alpha_0,0])=v_n([\alpha_0,0])=v'_n([\alpha_0,0])=0
$. 
Moreover let $\beta_0$ be the letter with $\pi_b(\beta_0)=1$ and observe that the identifications on 
$\cA\times\ZZ$ induced by Equation~\eqref{EquationIdentificationLabelsCriticalPoints} give  
$
[\alpha_0,-1]=[\beta_0,0]
$. 
Consider the label $[\chi,j]:=[\alpha_0,0]^\ast$, so that in particular $v_n([\chi,j])\to0$ as 
$n\to\infty$. 

If $\pi_t(\chi)=2$ and $j=1$, that is 
$
\underline{v([\alpha_0,0]^\ast)}
$ 
is the critical value of $T_\gamma$ right after $0$, then 
$
\underline{v([\beta_0,0]^\ast)}=\sup I^t_{\beta_0}(T_\gamma)
$. 
In this case condition $v_n([\alpha_0,0]^\ast)\to0$ as $n\to\infty$ means that the limit 
$
\widehat{F}\big(\tau(\cV)\big)=\lim_{n\to\infty}f_{\tau(\cV_n)}
$ 
has regular part $f_{\tau(\cV)}$ which does not contain $\beta_0$ in its alphabet, that is the intervals 
$
I^t_{\beta_0}(f_{\tau(\cV_n)})
$ 
and 
$
I^b_{\beta_0}(f_{\tau(\cV_n)})
$ 
vanish to a point. Therefore Equation~\eqref{EquartionLemmaCyclicBehavior} is satisfied, indeed we have
$$
v'_n([\beta_0,0]^\ast)-v'_n([\beta_0,0])
=
\sup I^t_{\beta_0}(f_{\tau(\cV_n)}) - \min I^t_{\beta_0}(f_{\tau(\cV_n)})
\to 0.
$$

In all other cases we have $[\beta_0,0]^\ast=[\chi,j-1]$, 
since $\underline{v([\chi,j])}$ belongs to the right-open interval 
$
[0,\lambda^{(\gamma)}_{\beta_0})=I^b_{\beta_0}(T_\gamma)
$ 
where $T_\gamma^{-1}$ acts continuously. The interval 
$
I_n:=\big(0,v_n([\chi,j])\big)
$ 
satisfies the assumption in Lemma~\ref{LemmaControlExpansionInverseBranches}, indeed we have 
$$
v_n([\alpha_0,0])=0<v_n([\chi,j])<\sup I^b_{\beta_0}(f_{\tau(\cV_n)}).
$$
Therefore Equation~\eqref{EquartionLemmaCyclicBehavior} follows from Lemma~\ref{LemmaControlExpansionInverseBranches}, since we have 
\begin{align*}
&
\big|
v'_n([\chi,j-1])-v'_n([\beta_0,0])
\big|
=
\big|
f_{\tau(\cV_n)}^{-1}\big(v_n([\chi,j])\big)
-
f_{\tau(\cV_n)}^{-1}\big(v_n([\alpha_0,0])\big)
\big|=
\\
&
\big|
f_{\tau(\cV_n)}^{-1}\big(v_n([\chi,j])\big)
-
f_{\tau(\cV_n)}^{-1}(0)
\big|
=
\big|
f_{\tau(\cV_n)}^{-1}\big(I_n\big)
\big|
\to0
\quad
\textrm{ as }
\quad
n\to\infty.
\end{align*}

\smallskip

\emph{Case (2): critical values.} Consider $[\alpha,i]$ with $\pi_b(\alpha)\geq 2$ and $i=1$, then consider the label $[\chi,j]:=[\alpha,1]^\ast$, so that 
$
v_n([\chi,j])
$ 
is collapsing onto the critical value $v_n([\alpha,1])$ of $f_{\tau(\cV_n)}$ at its left. One reasons as in Case (1). If 
$
\underline{v([\alpha,1]^\ast)}=\sup I^b_\alpha(T_\gamma)
$, 
then  
$
\underline{v([\alpha,0]^\ast)}=\sup I^t_\alpha(T_\gamma)
$ 
and Equation~\eqref{EquartionLemmaCyclicBehavior} becomes
$$
v'_n([\alpha_0,0]^\ast)-v'_n([\alpha_0,0])=
\sup I^t_\alpha(f_{\tau_n})-\min I^t_\alpha(f_{\tau_n})\to0,
$$
which is satisfied because the limit 
$
\widehat{F}\big(\tau(\cV)\big)=\lim_{n\to\infty}f_{\tau(\cV_n)}
$ 
has regular part $f_{\tau(\cV)}$ which does not contain $\alpha$ in its alphabet, that is the intervals labeled by the letter $\alpha$ are shrunk to a point. Otherwise, in all other cases one has $[\alpha,0]^\ast=[\chi,j-1]$, so that Equation~\eqref{EquartionLemmaCyclicBehavior} becomes 
$$
v'_n([\chi,j-1])-v'_n([\alpha,0])\to0,
$$
which follows from Lemma~\ref{LemmaControlExpansionInverseBranches}. Details are left to the reader.

\smallskip

\emph{Case (3): rightmost point of $\cV$.} Consider $[\alpha,i]:=[\alpha_m,i_m]$. Recall that in this case we have 
$
v([\alpha_m,i_m]^\ast)=
v'([\alpha_m,i_m]^\ast)=
v_n([\alpha_m,i_m]^\ast)=
v'_n([\alpha_m,i_m]^\ast)=1
$. 
Recall also from \S~\ref{SectionRauzyClasses} that we call $\alpha_b$ the letter with $\pi_b(\alpha_b)=d$ and let $\beta$ be the letter with $\pi_t(\beta)=\pi_t(\alpha_b)+1$. Observe that we have  
$
[\alpha_m,i_m-1]^\ast=[\beta,0]
$. 
We have $v_n([\alpha_m,i_m])\to1$ as $n\to\infty$, where 
$$
u^b_{\alpha_b}(f_{\tau(\cV_n)})=
v_n([\alpha_b,1])\leq
v_n([\alpha_m,i_m])<
1=\sup I^b_{\alpha_b}(f_{\tau(\cV_n)}),
$$
so that the interval $I_n:=\big(v_n([\alpha_m,i_m]),1\big)$ satisfies the assumption of Lemma~\ref{LemmaControlExpansionInverseBranches}, therefore Equation~\eqref{EquartionLemmaCyclicBehavior}  follows from such Lemma, since we have
\begin{align*}
&
\big|
v'_n([\alpha_m,i_m-1]^\ast)-v'_n([\alpha_m,i_m-1])
\big|
=
\big|
v'_n([\beta,0])-v'_n([\alpha_m,i_m-1])
\big|=
\\
&
\big|
f_{\tau(\cV_n)}^{-1}\big(v_n([\beta,1])\big)
-
f_{\tau(\cV_n)}^{-1}\big(v_n([\alpha_m,i_m])\big)
\big|
=
\big|
f_{\tau(\cV_n)}^{-1}\big(I_n\big)
\big|
\to0
\quad
\textrm{ as }
\quad
n\to\infty.
\end{align*}

\smallskip

\emph{Case (4): left neighborhood of critical values.} Consider $[\alpha,i]$ such that 
$
[\alpha,i]^\ast=[\chi,1]
$ 
for some $\chi$ with $\pi_b(\chi)\geq2$, then let $\beta$ be the letter with 
$
\pi_b(\beta)=\pi_b(\chi)-1
$, 
so that for the standard configuration $\cV^{(\gamma)}$ we have 
$
\underline{v([\alpha,i-1]^\ast)}=\sup I^t_\beta(T_\gamma)
$. 
In this case we have by assumption 
$
v_n([\chi,1])-v_n([\alpha,i])\to0 
$ 
as $n\to\infty$, where 
$$
u^b_\beta(f_{\tau(\cV_n)})=
v_n([\beta,1])\leq
v_n([\alpha,i])<
v_n([\chi,1])=\sup I^b_\beta(f_{\tau(\cV_n)}),
$$
and, reasoning as in Case (3), Equation~\eqref{EquartionLemmaCyclicBehavior} follows applying Lemma~\ref{LemmaControlExpansionInverseBranches} to the interval 
$
I_n:=\big(v_n([\alpha,i]),v_n([\chi,1])\big)
$, 
indeed we have 
\begin{align*}
&
\big|
v'_n([\alpha,i-1]^\ast)-v'_n([\alpha,i-1])
\big|
=
\\
&
\big|
\sup I^t_\beta(f_{\tau(\cV_n)})
-
f_{\tau(\cV_n)}^{-1}\big(v_n([\alpha,i])\big)
\big|
=
\big|
f_{\tau(\cV_n)}^{-1}\big(I_n\big)
\big|
\to0
\quad
\textrm{ as }
\quad
n\to\infty.
\end{align*}

\smallskip

\emph{Case (5): other points.} In all the remaining cases the labels $[\alpha,i]$ and 
$[\chi,j]:=[\alpha,i]^\ast$ correspond to points 
$
\underline{v([\alpha,i])}
$ 
and 
$
\underline{v([\chi,j])}
$ 
of the standard configuration $\cV^{(\gamma)}$ in the interior of the same continuity interval of 
$T_\gamma^{-1}$, hence we have 
$
[\alpha,i-1]^\ast=[\chi,j-1]
$. 
Therefore Equation~\eqref{EquartionLemmaCyclicBehavior} follows applying Lemma~\ref{LemmaControlExpansionInverseBranches} to the interval 
$
I_n:=\big(v_n([\alpha,i]),v_n([\chi,j])\big)
$. 
Details are left to the reader. 

\smallskip

The analysis of cases is complete and the Lemma is proved. 
\end{proof}

\end{document}